\documentclass[11pt,a4paper]{article}

\usepackage[T1]{fontenc}
\usepackage{mathpazo}
\usepackage[scaled=0.92]{helvet}
\usepackage{courier}
\usepackage{microtype}
\usepackage[a4paper,margin=2.65cm]{geometry}
\usepackage{amsmath,amssymb,amsthm,mathtools,bm}
\usepackage{booktabs,array,tabularx}
\usepackage{enumitem}
\usepackage{xcolor}
\usepackage[hidelinks,hypertexnames=false]{hyperref}
\usepackage[numbers,square,sort]{natbib}
\usepackage{fancyhdr}
\usepackage[nameinlink,noabbrev]{cleveref}

\hypersetup{
  pdfauthor={Leyang Wang},
  pdftitle={Vanishing boundary geometry and critical pressure compactness for the three-dimensional Navier--Stokes equations},
  pdfsubject={Boundary partial regularity for incompressible fluid flow},
  pdfkeywords={blow-up analysis; boundary partial regularity; incompressible flow; irregular domains; navier-stokes equations; sobolev multipliers},
  pdflang={en}
}

\setlist{leftmargin=2.2em,itemsep=0.15em,topsep=0.3em}
\allowdisplaybreaks[3]

\newtheorem{theorem}{Theorem}[section]
\newtheorem{proposition}[theorem]{Proposition}
\newtheorem{lemma}[theorem]{Lemma}
\newtheorem{corollary}[theorem]{Corollary}
\theoremstyle{definition}
\newtheorem{definition}[theorem]{Definition}

\theoremstyle{remark}
\newtheorem{remark}[theorem]{Remark}

\newcommand{\R}{\mathbb R}
\newcommand{\N}{\mathbb N}
\newcommand{\Hh}{\mathbb H}

\newcommand{\M}{\mathcal M}
\newcommand{\eps}{\varepsilon}
\newcommand{\dd}{\,\mathrm d}
\newcommand{\diver}{\operatorname{div}}

\newcommand{\dist}{\operatorname{dist}}

\newcommand{\Id}{\mathrm I}
\newcommand{\para}{\mathrm{para}}
\newcommand{\fintx}{\mathchoice{\mkern13mu\ooalign{$\displaystyle\int$\cr\hidewidth$\displaystyle-$\hidewidth\cr}}{\mkern7mu\ooalign{$\textstyle\int$\cr\hidewidth$\textstyle-$\hidewidth\cr}}{\mkern7mu\ooalign{$\scriptstyle\int$\cr\hidewidth$\scriptstyle-$\hidewidth\cr}}{\mkern7mu\ooalign{$\scriptscriptstyle\int$\cr\hidewidth$\scriptscriptstyle-$\hidewidth\cr}}}
\newcommand{\norm}[2]{\left\lVert #1\right\rVert_{#2}}

\title{Vanishing boundary geometry and critical pressure compactness\\
for the three-dimensional Navier--Stokes equations}
\author{Leyang Wang\\[0.4em]
\small Mathematics and Science College, Shanghai Normal University\\
\small 100 Guilin Road, Shanghai 200234, China\\[0.3em]
\small Corresponding author: \texttt{1000446516@smail.shnu.edu.cn}}
\date{}

\newif\ifmarked
\ifdefined\ShowChanges\markedtrue\else\markedfalse\fi
\definecolor{reviewone}{RGB}{0,75,170}
\definecolor{reviewtwo}{RGB}{130,35,115}
\definecolor{editorchange}{RGB}{0,105,75}
\newcommand{\RoneStart}{\ifmarked\color{reviewone}\fi}
\newcommand{\RtwoStart}{\ifmarked\color{reviewtwo}\fi}
\newcommand{\ReditorStart}{\ifmarked\color{editorchange}\fi}
\newcommand{\RevEnd}{\ifmarked\color{black}\fi}
\begin{document}

\maketitle
\ifmarked
\noindent{\small\textcolor{reviewone}{Blue: Reviewer 1}\quad
\textcolor{reviewtwo}{Purple: Reviewer 2}\quad
\textcolor{editorchange}{Green: editorial references}. Replaced passages are colored in full.}
\medskip
\fi

\begin{abstract}
We prove boundary partial regularity for the three-dimensional nonstationary
Navier--Stokes equations in bounded domains with fractional Sobolev boundary
graphs whose integrability exponent is greater than three. \RtwoStart
Within the fractional Sobolev graph setting, this lowers the boundary
threshold from fifteen quarters to three. We work in Breit's boundary suitable
weak-solution framework, with the mixed maximal-regularity pair
$(5/3,15/14)$ explicitly included in the solution assumptions.
\RevEnd The main difficulty is
to recover pressure compactness without imposing higher regularity on the
rough boundary coefficients. After a rigid rotation to the tangent plane and
parabolic rescaling, we show that each boundary graph becomes small in both
Sobolev multiplier classes needed for the local energy inequality and the
critical Stokes estimate. Geometry and fluid excess enter a contradiction
argument as independent small parameters, so the limiting problem is the
standard flat Stokes system. A localized pressure decomposition separates the
vanishing nonlinear and coefficient errors from the forced and homogeneous
flat pressures. The error pressure converges strongly at the critical mixed
integrability pair, whereas higher spatial integrability is used only for the
flat problems. These estimates yield a one-step excess decay bound with a
constant independent of the contraction radius. A physical-scale iteration,
followed by interior--boundary Campanato patching, gives local parabolic
H\"older continuity of the velocity at every regular boundary point when the
force has integrability exponent greater than five halves. The boundary
singular set is relatively closed and has zero parabolic Hausdorff measure in
dimension five thirds. The geometric decay determines a uniform starting
scale and does not restrict the final regularity exponent.
\end{abstract}



\section{Introduction}

The modern weak-solution theory of the three-dimensional Navier--Stokes
equations begins with Leray and Hopf \cite{Leray1934,Hopf1951}, while Serrin's
criterion is a classical benchmark for conditional regularity
\cite{Serrin1962}.  The interior
partial-regularity program was initiated by Scheffer
\cite{Scheffer1976,Scheffer1977,Scheffer1980} and reached its scale-invariant
form in the theorem of Caffarelli, Kohn and Nirenberg \cite{CKN}.  Subsequent
proofs and refinements clarified the blow-up, compactness and singular-set
mechanisms; see, among others,
\cite{Struwe1988,Lin1998,LadyzhenskayaSeregin1999,ChoeLewis,
GustafsonKangTsai2007,Vasseur2007}.

Boundary partial regularity is substantially more delicate than its interior
counterpart. Besides the pressure, which is nonlocal, one must control the
coefficients generated by flattening the boundary and transform the local
energy inequality without losing the natural energy scale.  The half-space
and smooth-boundary theory was developed through the works of Scheffer,
Seregin, Kang and their collaborators
\cite{Scheffer1982,Seregin2000,Seregin2002,Seregin2003,Kang2004,
GustafsonKangTsai2006,SereginShilkinSolonnikov,Wolf2010}; later endpoint and
near-endpoint boundary criteria include \cite{Barker2017}.  Breit
\cite{Breit2025} recently established an existence and boundary
partial-regularity theory in considerably rougher domains, using Sobolev
multipliers to encode the regularity of the boundary charts.

The linear analytic background combines half-space Stokes estimates and
maximal regularity \cite{Solonnikov1968,Solonnikov1977,Ukai1987,
GigaSohr1991,DeschHieberPruess2001,FarwigSohr1994,Solonnikov2003,Sohr2001}
with the
divergence-solving machinery of Bogovski\u{\i} and its negative-order extension
\cite{Bogovskii1979,GeissertHeckHieber2006}.  Estimates for the Stokes operator
on Lipschitz domains \cite{BrownShen1995} provide a useful comparison with the
multiplier-domain framework used here.  Local pressure formulations, although
not used as a substitute for the present Stokes decomposition, give a related
perspective on pressure localization \cite{Wolf2017}.

In Breit's theorem the boundary graph belongs locally to
\[
  W^{2-1/p,p}(\R^2), \qquad p>\frac{15}{4}.
\]
The number $15/4$ appears in the blow-up argument when the pressure gradient is
estimated in $L_t^{5/3}L_x^p$ and then converted into decay of an
$L^{5/3}$-based pressure excess. It is not the threshold at which the boundary
becomes differentiable. Indeed,
\[
 W^{2-1/p,p}(\R^2)\hookrightarrow C^{1,\,1-3/p}(\R^2)
 \quad\text{as soon as }p>3.
\]
\RoneStart
This suggests that the geometric threshold should be $p>3$, provided that the
pressure argument is reorganized.

The force assumption must also be compared. In \cite[Theorems 2.6--2.7,
arXiv version specified below]{Breit2025}, the force belongs to
$L^p(I;L^p(\Omega))$, with the same $p>15/4$ used in the boundary multiplier
condition. Here $P_b>3$ and $P_f>5/2$ are independent. In particular, keeping
any force exponent $p>15/4$ from that theorem is permitted: the reduction of
the boundary threshold does not require a stronger force assumption. The
range $5/2<P_f\le15/4$ is also permitted by our estimates. This comparison
concerns the fractional Sobolev graph setting; we do not assert inclusion of
Breit's full class of multiplier domains.

\begin{center}
\small
\begin{tabularx}{\linewidth}{@{}p{.19\linewidth}XX@{}}
\toprule
Assumption & Breit & Present theorem\\
\midrule
Boundary & Small Lipschitz and $\M^{2-1/p,p}$ chart norms, $p>15/4$
& Finite $W^{2-1/P_b,P_b}$ graph atlas, $P_b>3$; the two required small multiplier norms follow after localization\\
Force & $f\in L^p_{t,x}$ with the same $p>15/4$
& $f\in L^{P_f}_{t,x}$, $P_f>5/2$, independent of $P_b$\\
Solution framework & Boundary suitability with mixed maximal regularity
& The specified pair $(5/3,15/14)$ in that framework\\
\bottomrule
\end{tabularx}
\end{center}
\RevEnd

\RtwoStart
The present manuscript develops such a reorganization. In our decay lemma,
geometric smallness is an independent hypothesis, separate from the
normalization by the fluid excess and the force. This choice avoids a
nondecaying error bound that arises from the particular additive geometric
normalization discussed in \Cref{sec:dimension}; it is not an impossibility
claim about other choices of excess.
\RevEnd In a
contradiction sequence, the geometric threshold and the fluid amplitude are
sent to zero simultaneously. The limiting coefficients are therefore the
identity, and one can use flat Stokes estimates rather than a regularity theory
for a fixed perturbed limit system.

\RtwoStart
The main contributions are as follows.

The underlying fractional product estimates, Sobolev--Besov embeddings,
flattening construction, and Stokes maximal regularity are existing tools.
Our contribution is their quantitative combination at two distinct multiplier
levels and the curved--flat pressure separation that allows the higher
spatial exponents to be used only for flat problems. We do not claim a new
general multiplication theorem or a new general Stokes maximal-regularity
theory.
\RevEnd

\begin{enumerate}
  \item We prove that a tangent-plane normalized
  $W^{2-1/P_b,P_b}$ graph with $P_b>3$ becomes small at the rate
  $r^{1-3/P_b}$ in both multiplier spaces required by the argument:
  $\M^{4/3,3/2}$ for the local energy inequality and
  $\M^{16/15,15/14}$ for the critical pressure estimate.

  \item We formulate a one-step excess-decay lemma with the correct order of
  quantifiers. Its decay constant is independent of the contraction radius,
  whereas the geometric and fluid smallness thresholds are allowed to depend
  on that radius.

  \item We isolate the pressure mechanism at the critical pair
  $(5/3,15/14)$. After localization, the nonlinear and coefficient errors
  generate a pressure that is small by
  $W^{1,15/14}(\R^3_+)\hookrightarrow L^{5/3}(\R^3_+)$. The forced flat
  pressure decays with exponent $6-15/P_f$, while the flat homogeneous
  pressure decays with exponent $12/5-9/Q$.

  \item We show that the geometric exponent $1-3/P_b$ determines only the
  first scale at which the iteration can start. It does not restrict the
  final Campanato exponent. The bad-point covering remains at parabolic
  dimension $5/3$.
\end{enumerate}

Numbered results cited from \cite{Breit2025} refer to the openly available
version arXiv:2208.00415v2; the bibliography also records the final journal
publication.

The paper is organized as follows. \Cref{sec:setting} gives the equations,
the multiplier classes and the excess. \Cref{sec:multiplier} proves the
cross-exponent multiplier lemma. \Cref{sec:flattening} records the consequences
for the flattened coefficients. \Cref{sec:blowup} states and proves the
one-step decay lemma using the critical-pressure separation proved in
\Cref{sec:pressure}. The physical-scale iteration, interior--boundary
patching and singular-set estimate are given in
\Cref{sec:iteration,sec:singular}. Finally,
\RtwoStart
\Cref{sec:dimension} explains the limitation of the particular additive
geometric normalization considered here.
\RevEnd

\section{Setting and main statement}\label{sec:setting}

Let $I=(0,T)$ and let $\Omega\subset\R^3$ be bounded. We consider
\begin{equation}\label{eq:NS}
\begin{aligned}
 \partial_t u +(u\cdot\nabla)u-\Delta u+\nabla\pi &= f,
   &&\text{in }I\times\Omega,\\
 \diver u&=0,
   &&\text{in }I\times\Omega,\\
 u&=0,
   &&\text{on }I\times\partial\Omega.
\end{aligned}
\end{equation}
We use the parabolic cylinders
\[
  Q_r(z_0)=I_r(t_0)\times B_r(x_0),
  \qquad I_r(t_0)=(t_0-r^2,t_0),
\]
also denoted by $Q_r^-(z_0)$ when the backward orientation must be emphasized,
and write $Q_r^+=I_r\times B_r^+$ in flat coordinates. Averages are denoted
by $\fintx_E h=|E|^{-1}\int_Eh$.

\subsection{Sobolev multipliers}

For $s\ge1$ and $1<p<\infty$, the boundary multiplier norm is
\begin{equation}\label{eq:multiplier-definition}
 \norm{\varphi}{\M^{s,p}(\R^2)}
 :=\sup_{\norm{v}{W^{s-1,p}(\R^2)}=1}
 \norm{\nabla\varphi\cdot v}{W^{s-1,p}(\R^2)}.
\end{equation}
This is the form used in \cite{Breit2025,MazyaShaposhnikova}; only derivatives
of the boundary graph enter the transformed equations.

We distinguish the boundary exponent $P_b$ from the force exponent $P_f$:
\begin{equation}\label{eq:exponents}
  P_b>3, \qquad P_f>\frac52,
  \qquad \gamma_b:=1-\frac3{P_b}>0,
  \qquad \mu_f:=6-\frac{15}{P_f}>0.
\end{equation}
The two multiplier levels are
\begin{equation}\label{eq:two-levels}
 (s_1,p_1)=\left(\frac43,\frac32\right),
 \qquad
 (s_0,p_0)=\left(\frac{16}{15},\frac{15}{14}\right).
\end{equation}
\RtwoStart
The first level controls the geometric term in the local energy inequality;
the second is the trace level associated with
$W^{2,15/14}$ Stokes maximal regularity.

The word multiplier means that multiplication by the differentiated graph
acts boundedly on the stated fractional Sobolev space. It records the
coefficient control needed after changing coordinates, rather than a
pointwise bound on curvature; see \cite{MazyaShaposhnikova,Breit2025}.
The roles of the norms are summarized below.
\begin{center}
\small
\begin{tabularx}{\linewidth}{@{}p{.28\linewidth}XX@{}}
\toprule
Chart quantity & Analytical role & Where used\\
\midrule
$\M^{4/3,3/2}$ & Controls the second-derivative chart terms in the energy inequality
& Energy bound in \Cref{sec:blowup}\\
$\M^{16/15,15/14}$ & Gives $W^{1,15/14}$ coefficient multipliers after extension
& Critical Stokes and error-pressure estimates in \Cref{sec:pressure}\\
$\|\nabla\widehat\psi_r\|_\infty$ & Controls invertibility, ellipticity and convergence to the identity
& Flattening in \Cref{sec:flattening}\\
\bottomrule
\end{tabularx}
\end{center}
Here $P_b$ describes the boundary, $P_f$ describes the force, and
$(r_*,s_*)=(5/3,15/14)$ describes the mixed time--space maximal-regularity
class. The exponent $Q>15/4$ introduced later is used only for a homogeneous
flat Stokes problem. These exponents serve different purposes.
\RevEnd

\subsection{Boundary charts}

We initially assume that $\partial\Omega$ admits a finite graph atlas
$\{\varphi_\ell\}_{\ell=1}^N$ on disks of a common radius $2r_a$ such that
\begin{equation}\label{eq:graph}
 \varphi_\ell\in W^{2-1/P_b,P_b}(B_{2r_a}'),
 \qquad \Omega=\{(x',x_3):x_3>\varphi_\ell(x')\}
\end{equation}
\RtwoStart
in the corresponding orthonormal coordinates.

This is the local graph description used in the multiplier-domain approach
\cite[Section 2.4]{Breit2025}; fractional Sobolev graph regularity and the
associated multipliers are discussed in \cite{MazyaShaposhnikova}.
We assume
\RevEnd
\begin{equation}\label{eq:atlas-bound}
 \max_{1\le\ell\le N}
 \norm{\varphi_\ell}{W^{2-1/P_b,P_b}(B_{2r_a}')}
 \le M_b.
\end{equation}
At each $x_0\in\partial\Omega$ we then translate and apply a rigid rotation
$Q_{x_0}\in SO(3)$ so that $x_0$ becomes the origin and the tangent plane is
$\{x_3=0\}$. The uniform validity of this regraphing, including its
$W^{2-1/P_b,P_b}$ bound, is proved in
\Cref{lem:tangent-regraphing}. The rotation is essential: subtracting the
tangent plane by a nonorthogonal shear would leave a constant anisotropic
principal operator in the limit.

\subsection{The excess}

For a boundary cylinder, set
\begin{equation}\label{eq:excess}
\begin{aligned}
 E_r(u,\pi)
 &:={\fintx}_{Q_r^+}|u|^3
 \\
 &\quad
 +r^3\left(
 {\fintx}_{I_r}{\fintx}_{B_r^+}
 |\pi-(\pi)_{B_r^+}(t)|^{5/3}\dd x\dd t
 \right)^{9/5}.
\end{aligned}
\end{equation}
Under the Navier--Stokes scaling
\begin{equation}\label{eq:NS-scaling}
\begin{aligned}
 u_r(t,x)&=r u(t_0+r^2t,x_0+rx),\\
 \pi_r(t,x)&=r^2\pi(t_0+r^2t,x_0+rx),\\
 f_r(t,x)&=r^3f(t_0+r^2t,x_0+rx).
\end{aligned}
\end{equation}
one has
\begin{equation}\label{eq:excess-scaling}
 E_1(u_r,\pi_r)=r^3E_r(u,\pi),
 \qquad
 \norm{f_r}{L^{P_f}(Q_1^+)}^3
 =|Q_1^+|^{3/P_f}r^9
 \left({\fintx}_{Q_r^+}|f|^{P_f}\right)^{3/P_f}.
\end{equation}

\subsection{Boundary suitability and the main theorem}

We use the critical pair
\[
 r_*=\frac53,\qquad s_*=\frac{15}{14}.
\]
Write $W^{1,2}_{0,\diver}(\Omega)$ for the closure of compactly supported
solenoidal test fields in $W^{1,2}_0(\Omega)$ and
$W^{1,s_*}_\perp(\Omega)$ for the pressure space with a fixed spatial
normalization.

\begin{definition}[Boundary suitable weak solution]\label{def:suitable}
Let $f\in L^{r_*}(I;L^{s_*}(\Omega))$ and
$u_0\in W^{2,s_*}(\Omega)\cap W^{1,2}_{0,\diver}(\Omega)$. A pair $(u,\pi)$
is boundary suitable at the critical level if
\begin{align*}
 u&\in L^\infty(I;L^2(\Omega))\cap
 L^2(I;W^{1,2}_{0,\diver}(\Omega))
 \cap L^{r_*}(I;W^{2,s_*}(\Omega))
 \cap W^{1,r_*}(I;L^{s_*}(\Omega)),\\
 \pi&\in L^{r_*}(I;W^{1,s_*}_\perp(\Omega)),
\end{align*}
the equations and the initial condition $u(0)=u_0$ hold almost everywhere,
and the local energy inequality
\begin{align}\label{eq:local-energy}
 &\frac12\int_\Omega \zeta(t)|u(t)|^2\,\dd x
 +\int_0^t\!\int_\Omega \zeta|\nabla u|^2\,\dd x\dd\sigma\nonumber\\
 &\quad\le
 \int_0^t\!\int_\Omega
 \left\{\frac{|u|^2}{2}(\partial_t\zeta+\Delta\zeta)
 +\left(\frac{|u|^2}{2}+\pi\right)u\cdot\nabla\zeta
 +\zeta f\cdot u\right\}\dd x\dd\sigma
\end{align}
holds for every nonnegative $\zeta\in C_c^\infty(I\times\R^3)$ supported
spatially in a boundary chart and for almost every $t\in I$. In particular,
$\zeta$ vanishes near $t=0$, so no initial-energy term occurs above. This is Breit's Definition~2.4 specialized to
\RtwoStart
$(r_*,s_*)=(5/3,15/14)$.
\end{definition}
\RtwoStart

\begin{remark}[Scope of the solution class]\label{rem:solution-scope}
The maximal-regularity requirements on $\partial_tu$, $\nabla^2u$ and
$\nabla\pi$ above are explicit hypotheses. The theorem is not stated for
all suitable weak solutions specified only by energy bounds, pressure
integrability and a local energy inequality. The retained feature is the
pair $(5/3,15/14)$ within Breit's boundary suitable framework, not an
assertion that these extra derivatives follow from the usual suitable
weak-solution definition alone. \Cref{cor:existence} explains why the stated
class is nonempty under our data assumptions.
\end{remark}
\RevEnd

\begin{theorem}[Boundary partial regularity]\label{thm:main}
Let $\Omega\subset\R^3$ be a bounded domain satisfying
\eqref{eq:atlas-bound} for some $P_b>3$. Let $P_f>5/2$ and
$f\in L^{P_f}(I\times\Omega)$. Assume that $(u,\pi)$ is boundary suitable in
the sense of \Cref{def:suitable}. Then there is a relatively closed set
\[
 \Sigma\subset I\times\partial\Omega
\]
such that
\begin{equation}\label{eq:hausdorff-main}
 \mathcal H^{5/3}_{\para}(\Sigma)=0.
\end{equation}
Moreover, $u$ is locally parabolically H\"older continuous near every point of
$(I\times\partial\Omega)\setminus\Sigma$. More precisely, for every
\begin{equation}\label{eq:holder-range}
 0<\alpha<\min\left\{2-\frac5{P_f},\frac45\right\}
\end{equation}
one may choose an excess exponent $\beta$ satisfying
$3\alpha<\beta<\min\{6-15/P_f,12/5\}$, and the corresponding Campanato
estimate implies $u\in C^{0,\alpha}_{\para}$ locally at every regular
boundary point.
\end{theorem}

\begin{corollary}[Existence of a solution covered by \Cref{thm:main}]
\label{cor:existence}
Under the assumptions on $\Omega$, $P_b$, $P_f$ and $f$ in
\Cref{thm:main}, let
\[
 u_0\in W^{2,15/14}(\Omega)\cap W^{1,2}_{0,\diver}(\Omega).
\]
Then there exists a boundary suitable weak solution to which
\Cref{thm:main} applies.
\end{corollary}

\begin{proof}
On a bounded cylinder, $L^{P_f}(I\times\Omega)$ embeds into
$L^{5/3}(I;L^{15/14}(\Omega))$. By
\Cref{lem:tangent-regraphing,thm:cross-multiplier}, the boundary atlas can be
refined until both its Lipschitz constants and its
$\M^{16/15,15/14}$ norms lie below the threshold in
\cite[Theorem 2.5]{Breit2025}. That theorem gives a solution satisfying
\Cref{def:suitable}; the asserted partial regularity then follows from
\Cref{thm:main}.
\end{proof}

\section{The vanishing cross-exponent multiplier lemma}\label{sec:multiplier}

This section contains the geometric lemma that makes the range $P_b>3$
available simultaneously at the energy and pressure levels.

\begin{lemma}[Uniform tangent-plane regraphing]\label{lem:tangent-regraphing}
Assume \eqref{eq:graph}--\eqref{eq:atlas-bound} with $P_b>3$. After reducing
$r_a$ and increasing $M_b$ by constants depending only on the original finite
atlas, every $x_0\in\partial\Omega$ has an orthonormal coordinate system in
which
\begin{equation}\label{eq:tangent-graph}
 \Omega=\{(y',y_3):y_3>\psi_{x_0}(y')\},\qquad
 \psi_{x_0}(0)=0,\quad \nabla\psi_{x_0}(0)=0,
\end{equation}
and
\begin{equation}\label{eq:tangent-atlas-bound}
 \sup_{x_0\in\partial\Omega}
 \norm{\psi_{x_0}}{W^{2-1/P_b,P_b}(B_{2r_a}')}
 \le M_b.
\end{equation}
\end{lemma}

\begin{proof}
Put $\sigma_b=1-1/P_b$. Since
$W^{\sigma_b,P_b}(\R^2)\hookrightarrow C^{0,1-3/P_b}$, every original chart
is $C^{1,\gamma_b}$. Fix a point of one chart and rotate its tangent plane to
$\{y_3=0\}$. If
\[
 X(\xi')=(\xi',\varphi_\ell(\xi')),
\]
the first two components of the rotated map $X(\xi')-X(\xi'_0)$ define a
planar map $F$. Its derivative at $\xi'_0$ is invertible, and the uniform
H\"older modulus of $\nabla\varphi_\ell$ makes $F$ uniformly bi-Lipschitz on
a smaller disk. The new height function has the form $\psi=G\circ F^{-1}$.

The chain rule expresses $\nabla\psi$ as a smooth rational function of
$\nabla\varphi_\ell$, composed with $F^{-1}$; its denominator stays uniformly
away from zero. The Slobodeckij seminorm is stable under bi-Lipschitz changes
of variables, and $W^{\sigma_b,P_b}\cap L^\infty$ is stable under smooth
Lipschitz compositions. Hence
\[
 \norm{\nabla\psi}{W^{\sigma_b,P_b}}
 \le C\bigl(1+\norm{\nabla\varphi_\ell}{L^\infty}\bigr)^C
 \norm{\nabla\varphi_\ell}{W^{\sigma_b,P_b}}.
\]
The rotation sends the tangent plane to the horizontal plane, so
$\nabla\psi(0)=0$. Translation gives $\psi(0)=0$. Uniformity follows from the
finiteness of the original atlas.
\end{proof}

Let $\chi\in C_c^\infty(B_2')$ satisfy $\chi=1$ on $B_1'$. For a normalized
graph $\psi$ as in \eqref{eq:tangent-graph}, define
\begin{equation}\label{eq:scaled-chart}
 \psi_r(y'):=\frac{\psi(ry')}{r},
 \qquad \widehat\psi_r:=\chi\psi_r.
\end{equation}

\begin{lemma}[A cross-exponent fractional product estimate]
\label{lem:fractional-product}
Let $0<a<1$, $1<p<\infty$ with $ap<2$, let $\rho=2/a$, and let
$\varepsilon>0$. Then
\begin{equation}\label{eq:fractional-product}
 \norm{Fv}{W^{a,p}(\R^2)}
 \le C\left(\norm{F}{L^\infty(\R^2)}
       +\norm{F}{B^{a+\varepsilon}_{\rho,p}(\R^2)}\right)
       \norm{v}{W^{a,p}(\R^2)}
\end{equation}
for $F\in L^\infty\cap B^{a+\varepsilon}_{\rho,p}$ and
$v\in W^{a,p}(\R^2)$.
\end{lemma}

\begin{proof}
\RtwoStart
We give the dyadic estimates, using standard inhomogeneous
Littlewood--Paley blocks $\Delta_j$, $j\ge-1$, and low-frequency sums
$S_{j-1}$; see \cite{RunstSickel,Triebel}. Norm equivalence gives
$W^{a,p}=B^a_{p,p}$ because $0<a<1$. Put
$V=\|v\|_{B^a_{p,p}}$, and define $p_a$ by
$1/p_a=1/p-a/2>0$. Fractional Sobolev embedding and the choice
$\rho=2/a$ give
\[
 \|v\|_{L^{p_a}}\le CV,\qquad
 \frac1p=\frac1\rho+\frac1{p_a}.
\]
Bony's decomposition is $Fv=T_Fv+T_vF+R(F,v)$, where
$T_Fv=\sum_jS_{j-1}F\,\Delta_jv$ and the remainder sums pairs of comparable
frequencies. Finite overlap of annular Fourier supports gives, for a fixed
integer $N$,
\[
 2^{ka}\|\Delta_kT_Fv\|_p
 \le C\|F\|_\infty\sum_{|j-k|\le N}2^{ja}\|\Delta_jv\|_p.
\]
Taking the $\ell^p$ norm in $k$ controls this term by
$C\|F\|_\infty V$.
For the reversed paraproduct, the uniform $L^{p_a}$ boundedness of
$S_{j-1}$ and H\"older's inequality imply
\[
 2^{ka}\|\Delta_kT_vF\|_p
 \le CV\sum_{|j-k|\le N}2^{ja}\|\Delta_jF\|_\rho.
\]
Thus it is bounded by $CV\|F\|_{B^a_{\rho,p}}$ after summation.
For the remainder set $A_j=2^{ja}\|\Delta_jF\|_\rho$.
The products at level $j$ have Fourier support in a ball of radius
$C2^j$, so only $j\ge k-N$ contribute. The blocks of $v$ are bounded in
$L^{p_a}$ by $CV$, and hence
\[
 2^{ka}\|\Delta_kR(F,v)\|_p
 \le CV\sum_{j\ge k-N}2^{-(j-k)a}A_j.
\]
The discrete kernel on the right belongs to $\ell^1$ since $a>0$.
Young's inequality for sequences bounds its $\ell^p$ norm by
$CV\|A\|_{\ell^p}$. The finitely many lowest-frequency terms obey the
same bounds with a changed constant. This proves the stronger sufficient
estimate with $B^a_{\rho,p}$ in place of $B^{a+\varepsilon}_{\rho,p}$;
the stated result follows by the continuous inclusion of the latter space
into the former. One first argues for smooth truncated decompositions and
then passes to the limit in $W^{a,p}$.

In particular, a positive $\varepsilon$ is not needed for the remainder
summation itself. In the application below it is a convenient allowance
provided by the strict boundary Sobolev index gap when changing the
integrability and fine indices. This distinguishes the standard product
estimate from the quantitative boundary application.
\RevEnd
\end{proof}

\begin{theorem}[Vanishing cross-exponent Sobolev multiplier lemma]
\label{thm:cross-multiplier}
Let $P_b>3$ and let
$\psi\in W^{2-1/P_b,P_b}(B_{2r_a}')$ satisfy
$\psi(0)=0$ and $\nabla\psi(0)=0$. For $0<r<r_a$,
\begin{equation}\label{eq:cross-multiplier}
\begin{aligned}
 &\norm{\widehat\psi_r}{\M^{4/3,3/2}(\R^2)}
 +\norm{\widehat\psi_r}{\M^{16/15,15/14}(\R^2)}
 +\norm{\nabla\widehat\psi_r}{L^\infty(\R^2)}
 \\
 &\hspace{5em}\le
 C r^{\gamma_b}
 \norm{\psi}{W^{2-1/P_b,P_b}(B_{2r}')},
 \qquad \gamma_b=1-\frac3{P_b},
\end{aligned}
\end{equation}
where $C$ depends only on $P_b$ and $\chi$.
\end{theorem}

\begin{proof}
Set
\[
 s_b=2-\frac1{P_b},\qquad
 \sigma_b=s_b-1=1-\frac1{P_b}.
\]
Morrey's embedding in two dimensions gives
\[
 W^{\sigma_b,P_b}(B_{2r}')\hookrightarrow C^{0,\gamma_b}(B_{2r}'),
 \qquad \gamma_b=\sigma_b-\frac2{P_b}=1-\frac3{P_b}.
\]
Since $\nabla\psi(0)=0$,
\begin{equation}\label{eq:gradient-sup}
 \norm{\nabla\psi_r}{L^\infty(B_2')}
 =\sup_{B_2'}|\nabla\psi(ry')-\nabla\psi(0)|
 \le Cr^{\gamma_b}
 \norm{\nabla\psi}{W^{\sigma_b,P_b}(B_{2r}')}.
\end{equation}
The fractional seminorm scales exactly as
\begin{equation}\label{eq:fractional-scaling}
 [\nabla\psi_r]_{W^{\sigma_b,P_b}(B_2')}
 =r^{\sigma_b-2/P_b}
 [\nabla\psi]_{W^{\sigma_b,P_b}(B_{2r}')}
 =r^{\gamma_b}
 [\nabla\psi]_{W^{\sigma_b,P_b}(B_{2r}')}.
\end{equation}
Furthermore, $\psi(0)=0$ and \eqref{eq:gradient-sup} imply
$\norm{\psi_r}{L^\infty(B_2')}\le Cr^{\gamma_b}$ times the same local norm.
The product rule for a fixed smooth cutoff therefore yields
\begin{equation}\label{eq:scaled-W}
 \norm{\widehat\psi_r}{W^{2-1/P_b,P_b}(\R^2)}
 +\norm{\nabla\widehat\psi_r}{L^\infty(\R^2)}
 \le Cr^{\gamma_b}
 \norm{\psi}{W^{2-1/P_b,P_b}(B_{2r}')}.
\end{equation}

It remains to change both the differentiability and integrability indices.
Put $a_j=s_j-1$ and, for $j\in\{0,1\}$, define
\[
 \rho_j:=\frac{2}{a_j}.
\]
Thus $\rho_1=6$, $\rho_0=30$, and
\begin{equation}\label{eq:critical-Besov}
 \frac43-\frac26=1,
 \qquad
 \frac{16}{15}-\frac2{30}=1.
\end{equation}
On the other hand,
\begin{equation}\label{eq:source-Besov-index}
 s_b-\frac2{P_b}=2-\frac3{P_b}=1+\gamma_b>1.
\end{equation}
Fix $\varepsilon_b=\gamma_b/2$. The strict gap in
\eqref{eq:source-Besov-index}, together with compact support, gives the
strict Sobolev--Besov embeddings
\begin{equation}\label{eq:Besov-embeddings}
 B^{s_b}_{P_b,P_b}(\R^2)\hookrightarrow
 B^{4/3+\varepsilon_b}_{6,3/2}(\R^2)\cap
 B^{16/15+\varepsilon_b}_{30,15/14}(\R^2).
\end{equation}
Indeed, the source differential index is $1+\gamma_b$, whereas both target
indices are $1+\varepsilon_b$. If $P_b$ exceeds a target integrability
exponent, one first lowers integrability on the common bounded support; the
\RtwoStart
strict differentiability gap then permits the change of fine index. See
\cite{RunstSickel,Triebel}.

For clarity, if $P_b\le\rho_j$, the Sobolev change of integrability
first gives smoothness $a_j+\gamma_b$ for the gradient, and the decrease
to $a_j+\varepsilon_b$ permits any finite target fine index. If
$P_b>\rho_j$, compact support gives the decrease in integrability at
smoothness $\sigma_b$. This still leaves a positive smoothness gap:
\[
 \sigma_b-(a_1+\varepsilon_b)=\frac16+\frac1{2P_b}>0,
 \qquad
 \sigma_b-(a_0+\varepsilon_b)=\frac{13}{30}+\frac1{2P_b}>0.
\]
The change of fine index is consequently justified in both cases.
\RevEnd

Apply \Cref{lem:fractional-product} componentwise to
$F=\nabla\widehat\psi_r$, with $(a,p,\rho)=(a_j,p_j,\rho_j)$.
Since differentiation maps
$B^{s_j+\varepsilon_b}_{\rho_j,p_j}$ into
$B^{a_j+\varepsilon_b}_{\rho_j,p_j}$, the definition
\eqref{eq:multiplier-definition} gives
\[
 \norm{\widehat\psi_r}{\M^{s_j,p_j}(\R^2)}
 \le C\left(
 \norm{\nabla\widehat\psi_r}{L^\infty}
 +\norm{\widehat\psi_r}
 {B^{s_j+\varepsilon_b}_{\rho_j,p_j}}\right).
\]
Combining this bound with
\eqref{eq:scaled-W}--\eqref{eq:Besov-embeddings} proves
\eqref{eq:cross-multiplier}. \RtwoStart
The strict gap $\varepsilon_b<\gamma_b$ is used for the cross-exponent
Besov embeddings; the remainder estimate in \Cref{lem:fractional-product}
uses $a_j>0$. No endpoint multiplier criterion is invoked.
\RevEnd
\end{proof}

\begin{corollary}[Uniform geometric threshold]\label{cor:uniform-threshold}
Under \eqref{eq:atlas-bound}, define
\begin{equation}\label{eq:eta}
 \eta_{x_0}(r):=
 \norm{\widehat\psi_{x_0,r}}{\M^{4/3,3/2}}
 +\norm{\widehat\psi_{x_0,r}}{\M^{16/15,15/14}}
 +\norm{\nabla\widehat\psi_{x_0,r}}{L^\infty}.
\end{equation}
Then
\begin{equation}\label{eq:uniform-eta}
 \sup_{x_0\in\partial\Omega}\eta_{x_0}(r)
 \le CM_b r^{\gamma_b}.
\end{equation}
Consequently, for every $\delta_*>0$ there is
\begin{equation}\label{eq:R0}
 R_0\le \min\left\{r_a,
 \left(\frac{\delta_*}{CM_b}\right)^{1/\gamma_b}\right\}
\end{equation}
such that $\eta_{x_0}(r)\le\delta_*$ for every $x_0\in\partial\Omega$ and
$0<r<R_0$.
\end{corollary}

\begin{proof}
The estimate \eqref{eq:uniform-eta} follows immediately from
\Cref{lem:tangent-regraphing,thm:cross-multiplier} and
\eqref{eq:tangent-atlas-bound}; \eqref{eq:R0} then gives the asserted
threshold.

For later use in the existence argument, we also record how this normalized
estimate returns to physical coordinates. Define
\[
 \widehat\psi^{\,\mathrm{phys}}_{x_0,r}(x')
 :=r\,\widehat\psi_{x_0,r}(x'/r).
\]
This function agrees with the tangent-plane graph near $x_0$ and has compact
support in the physical chart. At the critical Besov indices in
\eqref{eq:critical-Besov}, homogeneous rescaling is invariant. With the
positive margin $\varepsilon_b=\gamma_b/2$ used above, the rescaling factor
is $r^{-\varepsilon_b}$; hence
\[
 \norm{\widehat\psi^{\,\mathrm{phys}}_{x_0,r}}
 {\M^{s_j,p_j}(\R^2)}
 +\norm{\nabla\widehat\psi^{\,\mathrm{phys}}_{x_0,r}}{L^\infty}
 \le C M_b r^{\gamma_b-\varepsilon_b}
 =C M_b r^{\gamma_b/2}.
\]
Thus a finite physical atlas can also be refined until its graph
multiplier norms and Lipschitz constants are below any prescribed threshold.
\end{proof}

\section{Flattening and convergence of the coefficients}\label{sec:flattening}

Let $T$ be the boundary extension operator used in \cite{Breit2025}. Applied
to $\widehat\psi_r$, it produces an extension in the half-space satisfying
\begin{equation}\label{eq:extension}
 \norm{T\widehat\psi_r}{\M^{2,p}(\Hh)}
 \le C\norm{\widehat\psi_r}{\M^{2-1/p,p}(\R^2)},\qquad
 \norm{\nabla T\widehat\psi_r}{W^{1,P_b}(\Hh)}
 \le C\norm{\nabla\widehat\psi_r}
 {W^{1-1/P_b,P_b}(\R^2)}.
\end{equation}
Fix $N$ larger than the uniform Lipschitz constant required in the extension
construction and set
\[
 \Phi_r(\xi',\xi_3)
 :=\left(\xi',\,\xi_3+
 (T\widehat\psi_r)(\xi',\xi_3/N)\right),\qquad
 \Psi_r:=\Phi_r^{-1},\qquad J_r:=\det\nabla\Phi_r .
\]
For $r<R_0$, \Cref{cor:uniform-threshold} makes $\Phi_r$ uniformly
bi-Lipschitz and $J_r$ uniformly separated from zero. The transformed
momentum and divergence operators contain
\begin{equation}\label{eq:coefficients-piola}
 A_r=J_r(\nabla\Psi_r\circ\Phi_r)^\top
       (\nabla\Psi_r\circ\Phi_r),
 \qquad
 B_r=J_r(\nabla\Psi_r\circ\Phi_r).
\end{equation}
With the matrix-gradient and divergence conventions of
\cite[Section 4]{Breit2025}, this is the appropriate cofactor matrix and the
Piola identity gives
\begin{equation}\label{eq:piola}
 \diver B_r=0
 \qquad\text{in }\mathcal D'(\Hh).
\end{equation}
In particular, $\diver(B_rc(t))=0$ for every spatially constant pressure
$c(t)$, and
$B_r^\top:\nabla w=\diver(B_rw)$ for compactly supported vector fields $w$.
For the present proof the following convergence consequence is decisive.

For a matrix field $C$ and $1<p<\infty$, write
\begin{equation}\label{eq:coefficient-norm}
 \norm{C}{\mathfrak X(p)}:=\norm{C}{L^\infty}
 +\sup_{\norm{w}{W^{1,p}}=1}\norm{Cw}{W^{1,p}}.
\end{equation}

\begin{lemma}[Coefficient convergence]\label{lem:coefficients}
For each of the two levels in \eqref{eq:two-levels}, there is $C>0$ such that
\begin{equation}\label{eq:coefficient-convergence}
 \norm{J_r-1}{L^\infty}
 +\norm{A_r-\Id}{\mathfrak X(p_j)}
 +\norm{B_r-\Id}{\mathfrak X(p_j)}
 \le C\eta_{x_0}(r),
\end{equation}
where the matrix norm is understood entrywise. These are exactly the
$L^\infty$ and $W^{1,p_j}$ multiplier bounds used to estimate the three
coefficient errors in the transformed Stokes system. In particular, the
left-hand side tends to zero uniformly in $x_0$ as $r\downarrow0$.
\end{lemma}

\begin{proof}
The extension estimate \eqref{eq:extension} gives
\[
 \norm{T\widehat\psi_r}{\M^{2,p_j}(\Hh)}
 +\norm{\nabla T\widehat\psi_r}{L^\infty(\Hh)}
 \le C\eta_{x_0}(r).
\]
The entries of $\nabla\Phi_r-\Id$ are first derivatives of
$T\widehat\psi_r$. If their $L^\infty$ norm is below a universal threshold,
the inverse Jacobian is given by a convergent Neumann series in
$\nabla\Phi_r-\Id$. The $W^{1,p_j}$ multiplier product rule therefore gives
\[
 \norm{\nabla\Psi_r\circ\Phi_r-\Id}{\mathfrak X(p_j)}
 +\norm{J_r-1}{L^\infty}
 \le C\eta_{x_0}(r).
\]
Substitution into the formulas for $A_r$ and $B_r$, followed by the same
multiplier product rule, proves \eqref{eq:coefficient-convergence}. The
uniform decay follows from \Cref{cor:uniform-threshold}.
\end{proof}

\begin{remark}[Why a rigid rotation is part of the lemma]
The map $x_3\mapsto x_3-\nabla\psi(x_0')\cdot x'$ is a shear, not an
orthogonal transformation. If it is used to remove the tangent plane, the
limit operator is $\diver(G_{x_0}\nabla u)$ with a constant matrix
$G_{x_0}\ne\Id$. By first rotating with $Q_{x_0}\in SO(3)$, the
Navier--Stokes operator is unchanged and the limiting principal part is
exactly $-\Delta$.
\end{remark}

\section{One-step excess decay and the two-parameter limit}\label{sec:blowup}

Fix $Q>15/4$ and define
\begin{equation}\label{eq:muQ}
 \mu_Q:=\frac{12}{5}-\frac9Q>0.
\end{equation}

\begin{lemma}[One-step boundary decay]\label{lem:one-step}
Let
\begin{equation}\label{eq:theta-range}
 0<\theta<\min\{\mu_f,\mu_Q\}.
\end{equation}
There exist $C_0>0$ and $\tau_0\in(0,1/8)$, independent of $\tau$, such that
for every $\tau\in(0,\tau_0)$ there are
$\delta_*(\tau)>0$ and $\eps_*(\tau)>0$ with the following property.
Suppose that the flattened cylinder comes from a tangent-plane normalized
chart as in \eqref{eq:scaled-chart}, with the structural
$W^{2-1/P_b,P_b}$ bound inherited from
\eqref{eq:tangent-atlas-bound}. If a boundary suitable weak solution in
$Q_1^+$ satisfies
\begin{equation}\label{eq:one-step-smallness}
 \eta(1)\le\delta_*(\tau),
 \qquad
 E_1(u,\pi)+\norm{f}{L^{P_f}(Q_1^+)}^3\le\eps_*(\tau),
\end{equation}
then
\begin{equation}\label{eq:one-step}
 E_\tau(u,\pi)
 \le C_0\tau^\theta
 \left(E_1(u,\pi)+\norm{f}{L^{P_f}(Q_1^+)}^3\right).
\end{equation}
\end{lemma}

\begin{proof}
\RtwoStart
Choose once and for all $\nu$ with
$\theta<\nu<\min\{\mu_f,12/5\}$. The flat decay and pressure estimates
then have constants depending on the fixed exponents but not on $\tau$.
Fix $C_0$ larger than twice the combined constant used in
\eqref{eq:limsup-decay}.
\RevEnd Suppose that, for some fixed
$\tau\in(0,\tau_0)$, no pair $(\delta_*,\eps_*)$ exists. Then for every
$m\in\N$ there is a counterexample $(u_m,\pi_m,\varphi_m,f_m)$ satisfying
\begin{equation}\label{eq:double-smallness}
 \eta(\varphi_m)\le\frac1m,
 \qquad
 \lambda_m^3:=E_1(u_m,\pi_m)+\norm{f_m}{L^{P_f}}^3\le\frac1m,
\end{equation}
but
\begin{equation}\label{eq:failure}
 E_\tau(u_m,\pi_m)>C_0\tau^\theta\lambda_m^3.
\end{equation}
In addition to \eqref{eq:double-smallness}, the charts $\varphi_m$ retain
the common structural $W^{2-1/P_b,P_b}$ bound from the statement.
Normalize
\begin{equation}\label{eq:normalization}
 v_m=\frac{u_m}{\lambda_m},\qquad
 q_m=\frac{\pi_m-(\pi_m)_{B_1^+}(t)}{\lambda_m},\qquad
 g_m=\frac{f_m}{\lambda_m}.
\end{equation}
Then
\begin{equation}\label{eq:normalized-bound}
 {\fintx}_{Q_1^+}|v_m|^3
 +\left({\fintx}_{Q_1^+}|q_m|^{5/3}\right)^{9/5}
 +\norm{g_m}{L^{P_f}}^3=1.
\end{equation}
The transformed equations have the form
\begin{equation}\label{eq:normalized-system}
 J_m\partial_t v_m
 +\lambda_m(B_m\nabla v_m)v_m
 -\diver(A_m\nabla v_m)+\diver(B_mq_m)=J_mg_m,
 \qquad B_m^\top:\nabla v_m=0.
\end{equation}

At the energy multiplier level $(4/3,3/2)$, the transformed local energy
inequality and \eqref{eq:normalized-bound} give
\begin{equation}\label{eq:energy-bound}
 v_m\ \text{bounded in }\
 L_t^\infty L_x^2(Q_{3/4}^+)\cap L_t^2W_x^{1,2}(Q_{3/4}^+).
\end{equation}
For completeness, choose a cutoff $\zeta$ equal to one on $Q_{3/4}^+$ and
use $\zeta^2$ in the transformed local energy inequality. Uniform
ellipticity of $A_m$, positivity of $J_m$, and
\eqref{eq:normalized-bound} control the cutoff, pressure, force, and
$\lambda_m|v_m|^3$ terms. The only term containing a second derivative of
the inverse chart has the form
\[
 \int |v_m|\,|\zeta v_m\,\nabla^2\Psi_m| .
\]
The extension bound \eqref{eq:extension}, the
$\M^{4/3,3/2}$ product estimate, and Young's inequality give, for every
$\varepsilon>0$,
\[
 \int |v_m|\,|\zeta v_m\,\nabla^2\Psi_m|
 \le C_\varepsilon+
 \varepsilon\norm{\zeta\nabla v_m}{L^2(Q_1^+)}^2.
\]
Taking $\varepsilon$ below the ellipticity constant proves
\eqref{eq:energy-bound}. This is the energy-level computation in the proof
of \cite[Lemma~5.1]{Breit2025}; importantly, it uses only the
$\M^{4/3,3/2}$ norm, which is uniformly small here.

Set $w_m:=J_mv_m$. Since $J_m$ is independent of time, the momentum equation
directly controls $\partial_tw_m$. Testing
\eqref{eq:normalized-system} against $W^{1,15}_0(B_{3/4}^+)$ functions and
using the bounds for $A_m,B_m$ together with \eqref{eq:energy-bound} yields
\[
 \partial_t w_m\quad\text{bounded in}\quad
 L_t^{5/4}W_x^{-1,15/14}(Q_{3/4}^+).
\]
Indeed, the viscous term is controlled by $\nabla v_m\in L^2$, while the
pressure and convection terms lie in the indicated negative space because
$q_m\in L^{5/3}$ and the energy interpolation gives
$v_m\in L^{10/3}$. Moreover, $J_m\in W^{1,P_b}\cap L^\infty$ locally and
$P_b>3$, so H\"older and $W^{1,2}\hookrightarrow L^6$ show that $w_m$ is
bounded in $L_t^2W_x^{1,2}$. Aubin--Lions gives strong convergence of $w_m$
in $L^2$. Since $J_m\to1$ uniformly, the same is true for $v_m$.
Interpolation with the uniform $L^{10/3}$ bound then gives, after passing to
a subsequence,
\begin{equation}\label{eq:velocity-strong}
 v_m\to v\quad\text{strongly in }L^3(Q_{2/3}^+).
\end{equation}
At the same time,
\[
 q_m\rightharpoonup q\quad\text{in }L^{5/3}(Q_{2/3}^+),
 \qquad
 g_m\rightharpoonup g\quad\text{in }L^{P_f}(Q_{2/3}^+).
\]
By \eqref{eq:double-smallness} and \Cref{lem:coefficients},
\begin{equation}\label{eq:flat-coefficients}
 J_m\to1,\qquad A_m\to\Id,\qquad B_m\to\Id
\end{equation}
in all coefficient norms needed at both multiplier levels. Since
$\lambda_m\to0$, the convective term disappears. Hence the weak limit solves
the flat forced Stokes system
\begin{equation}\label{eq:flat-limit}
 \partial_t v-\Delta v+\nabla q=g,\qquad
 \diver v=0\quad\text{in }Q_{2/3}^+,\qquad
 v=0\quad\text{on }\{x_3=0\}.
\end{equation}

The pressure cannot be handled by weak convergence alone. Apply
\Cref{prop:pressure-separation} to write, in $Q_{1/2}^+$,
\begin{equation}\label{eq:pressure-split}
 q_m=q_m^{\rm c}+q_m^f+q_m^h+c_m(t),
\end{equation}
where $c_m(t)$ is irrelevant to the excess. The nonlinear and
coefficient-error pressures are combined in $q_m^{\rm c}$ and converge
strongly to zero in $L^{5/3}$, while $q_m^f$ and $q_m^h$ satisfy the uniform gradient bounds at
exponents $P_f$ and $(5/3,Q)$, respectively. Together with
\eqref{eq:velocity-strong} and \Cref{lem:flat-decay}, this gives, after
using the fixed choice of $\nu$ made at the beginning of the proof,
\begin{equation}\label{eq:limsup-decay}
 \limsup_{m\to\infty}E_\tau(v_m,q_m)
 \le C\big(\tau^\nu+\tau^{\mu_f}+\tau^{\mu_Q}\big)
 \le C\tau^\theta.
\end{equation}
The choice $C_0>2C$ made above, with $\tau_0$ decreased if necessary, contradicts
\eqref{eq:failure}. This proves the lemma.
\end{proof}

\begin{remark}[Quantifier order]
The thresholds $\delta_*$ and $\eps_*$ may depend on the fixed contraction
factor $\tau$, but $C_0$ must not. In the iteration one first chooses
$\theta$, then an exponent $\beta<\theta$, then $\tau$ so that
$C_0\tau^{\theta-\beta}<1$, and only afterwards fixes the two smallness
thresholds. \RtwoStart
Writing $C_0=C_0(\tau)$ would not provide a contraction.
\end{remark}
\RtwoStart

\begin{remark}[Independent smallness parameters]\label{rem:independent}
There is no hypothesis on the quotient $\eta(\varphi_m)/\lambda_m$.
Indeed, subtracting the flat linear operator after the normalization
\eqref{eq:normalization} produces $(1-J_m)\partial_tv_m$,
$\diver((A_m-\Id)\nabla v_m)$ and
$\diver((\Id-B_m)q_m)$. The geometric factors multiply the already
normalized unknowns; they are not divided by $\lambda_m$. Uniform critical
regularity therefore bounds these terms by $C\kappa_m$, while convection
is bounded by $C\lambda_m$. This is the precise sense of independence
used here; both parameters tend to zero, but no relative rate is required.
\end{remark}
\RevEnd

\section{Critical pressure separation}\label{sec:pressure}

The pressure is separated only after the localized equation has been rewritten
as a flat Stokes system. This avoids any request for
$L_t^{5/3}W_x^{1,Q}$ regularity of a Stokes system with rough coefficients.
\RtwoStart
The coefficient errors are estimated at the critical pair and then disappear through
the following homogeneous Sobolev embedding (see
\cite{Sohr2001,Triebel}):
\RevEnd
\begin{equation}\label{eq:critical-embedding}
 \dot W^{1,15/14}(\Hh)\hookrightarrow L^{5/3}(\Hh).
\end{equation}

We record the precise flat theorem used in the decomposition. It is the
zero-initial-data case of the nonzero-divergence half-space theorem
\cite{FarwigSohr1994,Solonnikov2003}; see also the flat estimate invoked in
\cite[Lemma 4.2]{Breit2025}.

\begin{lemma}[Flat Stokes system with nonzero divergence]
\label{lem:flat-nonzero-divergence}
Let $1<a,p<\infty$ and let $\mathcal C=\mathcal I\times\Hh$. Suppose that
$F\in L^a(\mathcal I;L^p(\Hh))$ and
\[
 D\in L^a(\mathcal I;W^{1,p}(\Hh)\cap L^p_\perp(\Hh)),\qquad
 \partial_tD\in L^a(\mathcal I;\dot W^{-1,p}(\Hh)),
\]
where $D$ is compactly supported in space, has zero spatial mean for almost
every time, and has zero trace at the lower endpoint $t_-$ of $\mathcal I$.
For compactly supported data there is a solution, modulo spatially constant
pressures, of
\[
 \partial_tU-\Delta U+\nabla P=F,\qquad
 \diver U=D,
\]
with zero initial value and zero trace on $\partial\Hh$, such that
$\partial_tU,\nabla^2U,\nabla P\in L^a(\mathcal I;L^p(\Hh))$ and
$U\in L^a(\mathcal I;L^p(K))$ for every compact
$K\Subset\overline\Hh$. It satisfies
\begin{align}\label{eq:flat-maxreg}
 &\norm{\partial_tU}{L^aL^p}
 +\norm{\nabla^2U}{L^aL^p}
 +\norm{\nabla P}{L^aL^p}\nonumber\\
 &\qquad\le C_{a,p}\left(
 \norm{F}{L^aL^p}+\norm{\nabla D}{L^aL^p}
 +\norm{\partial_tD}{L^a\dot W^{-1,p}}\right).
\end{align}
The velocity is unique in this local-$L^p$ homogeneous maximal-regularity
class, and the pressure is unique modulo functions of time. For $D=0$, the
solution operators are also consistent on intersections of admissible forcing
spaces: if $F$ belongs to the classes for both $(a,p)$ and
$(\widetilde a,\widetilde p)$, the two velocities coincide and the two
pressure gradients coincide.
\end{lemma}

\begin{proof}
The existence, estimate, and uniqueness are the zero-initial-data
nonzero-divergence estimate of
\cite{FarwigSohr1994,Solonnikov2003}, in the compactly supported half-space
form used in \cite[Lemma~4.2]{Breit2025}. For the last assertion, first take
a smooth compactly supported force in the intersection. The half-space Stokes
kernel representation is independent of the exponents, so the two
realizations agree. Approximation in both forcing norms and
\eqref{eq:flat-maxreg} then prove consistency for general intersection data.
\RtwoStart
To specify the uniqueness class, the zero initial trace and the global
bound on $\partial_tU$ also imply
\[
 \|U\|_{L^a(\mathcal I;L^p(\Hh))}
 \le |\mathcal I|\,\|\partial_tU\|_{L^a(\mathcal I;L^p(\Hh))}.
\]
Indeed, integrate $\partial_tU$ from the lower endpoint, first locally
and then over the half-space, and use Young's inequality in time.
Thus the solutions used here belong to the usual spatially integrable
zero-initial-data realization; the local-$L^p$ wording does not admit
additional nonintegrable homogeneous velocities.
\RevEnd
\end{proof}

\begin{remark}[Compatibility in the half-space]\label{rem:halfspace-compatibility}
The hypotheses used below are deliberately stronger than the minimum needed
in some homogeneous half-space formulations: every divergence datum is the
divergence of a compactly supported flux with zero boundary trace. Thus its
mean vanishes, its initial trace is zero, and it can be handled either by the
half-space theorem or by the corresponding localized half-ball estimate. No
flux at spatial infinity is used in the proof.
\end{remark}

\begin{proposition}[Critical curved--flat pressure separation]
\label{prop:pressure-separation}
Let $(v_m,q_m)$ solve \eqref{eq:normalized-system} and suppose that
\eqref{eq:normalized-bound}, \eqref{eq:energy-bound}, and
\eqref{eq:flat-coefficients} hold. Let $g_m$ be bounded in
$L^{P_f}(Q_{3/4}^+)$, with $P_f>5/2$. For every finite $Q>15/4$, after
subtracting time-dependent spatial means, the pressure admits in $Q_{1/2}^+$
the decomposition \eqref{eq:pressure-split} such that
\begin{align}
 \norm{q_m^{\rm c}}{L^{5/3}(Q_{1/2}^+)}&\longrightarrow0,
 \label{eq:critical-small}\\
 \norm{\nabla q_m^f}{L^{P_f}(Q_{1/2}^+)}&\le C,
 \label{eq:forced-pressure}\\
 \norm{\nabla q_m^h}{L_t^{5/3}L_x^Q(Q_{1/2}^+)}&\le C_Q.
 \label{eq:hom-pressure}
\end{align}
\RtwoStart
Here $q_m^{\rm c}$ contains both the vanishing nonlinear pressure and all
coefficient-error pressures. The constants are uniform for coefficient
triples in a fixed sufficiently small critical multiplier ball.
\end{proposition}
\RtwoStart

The three pressures have different roles. Only the critical error part
is asserted to converge strongly. The other two parts need uniform decay
estimates, not strong compactness.
\begin{center}
\small
\begin{tabularx}{\linewidth}{@{}p{.14\linewidth}XX@{}}
\toprule
Part & Flat data (notation defined below) & Estimate used\\
\midrule
$q_m^{\rm c}$ & Nonlinearity and coefficient errors; divergence $H_m^{\rm c}$
& $\|q_m^{\rm c}\|_{5/3}\to0$\\
$q_m^f$ & Localized force $\zeta J_mg_m$; zero divergence
& $\nabla q_m^f\in L^{P_f}_{t,x}$ uniformly\\
$q_m^h$ & Cutoff commutators $(K_m,H_m^{\rm a})$, vanishing on the inner cylinder
& $\nabla q_m^h\in L_t^{5/3}L_x^Q$ uniformly\\
\bottomrule
\end{tabularx}
\end{center}
All three auxiliary problems have zero initial data and no-slip boundary
data. The third is homogeneous only on the smaller cylinder, not on the
whole half-space.
\RevEnd

\begin{proof}
Set
\[
 a_0=\frac53,\qquad p_0=\frac{15}{14}.
\]
All mixed norms in the proof are over fixed cylinders, so changing their
radii only changes harmless constants.

\smallskip
\noindent\emph{Step 1: uniform critical regularity.}
Interpolation between the two energy spaces in
\eqref{eq:energy-bound} gives
\begin{equation}\label{eq:energy-interpolation}
 \norm{v_m}{L_t^{10}L_x^{30/13}(Q_{3/4}^+)}\le C.
\end{equation}
Consequently,
\begin{equation}\label{eq:nonlinear-forcing}
 \norm{(B_m\nabla v_m)v_m}
 {L_t^{a_0}L_x^{p_0}(Q_{3/4}^+)}
 \le C
 \norm{\nabla v_m}{L_t^2L_x^2}
 \norm{v_m}{L_t^{10}L_x^{30/13}}
 \le C.
\end{equation}
Indeed, $1/a_0=1/2+1/10$ and
$1/p_0=1/2+13/30$. Since $P_f>5/2$, the force is bounded in
$L_t^{a_0}L_x^{p_0}$ on the fixed cylinder. The local perturbed Stokes
estimate \cite[Lemma 4.3, with $p=q=p_0$]{Breit2025}, applied to
\eqref{eq:normalized-system}, therefore yields
\begin{equation}\label{eq:critical-local-maxreg}
\begin{aligned}
 &\norm{\partial_t v_m}{L_t^{a_0}L_x^{p_0}(Q_{2/3}^+)}
 +\norm{\nabla^2v_m}{L_t^{a_0}L_x^{p_0}(Q_{2/3}^+)}
 +\norm{\nabla q_m}{L_t^{a_0}L_x^{p_0}(Q_{2/3}^+)}
 \le C.
\end{aligned}
\end{equation}
The right-hand side is uniform because the lower-order norms of
$\nabla v_m$ and $q_m-(q_m)_{B_{3/4}^+}(t)$ are controlled by
\eqref{eq:energy-bound} and \eqref{eq:normalized-bound}. Only the
$\M^{16/15,15/14}$ smallness is used in
\eqref{eq:critical-local-maxreg}.

\smallskip
\noindent\emph{Step 2: localization in the flat half-space.}
Choose a space--time cutoff $\zeta$ which is one on $Q_{5/8}^+$, is
supported in $I_{2/3}\times B_{2/3}$, and vanishes near the lower endpoint
of the fixed interval $\mathcal I=I_{2/3}$. Subtract
$(q_m)_{B_{2/3}^+}(t)$ and set
\[
 U_m=\zeta v_m,\qquad
 P_m=\zeta\big(q_m-(q_m)_{B_{2/3}^+}(t)\big).
\]
Extend these functions by zero in the spatial directions to
$\mathcal C=\mathcal I\times\Hh$. Time-dependent pressure constants do not
alter the transformed equation because $\diver B_m=0$ by
\eqref{eq:piola}. The time cutoff is chosen to vanish on a neighborhood of
the lower endpoint of $\mathcal I$, so all divergence data introduced below
have zero initial trace.

Let $\mathcal L_m$ denote the momentum operator on the left of
\eqref{eq:normalized-system}, without convection, and define the cutoff
commutator
\[
 K_m:=\mathcal L_m(U_m,P_m)
 -\zeta\mathcal L_m(v_m,q_m).
\]
Also put
\begin{equation}\label{eq:cutoff-divergence}
 H_m^{\rm a}:=B_m^\top:\nabla U_m
 =B_m^\top:(v_m\otimes\nabla\zeta).
\end{equation}
Both $K_m$ and $H_m^{\rm a}$ are supported where a derivative of $\zeta$
is nonzero. The product estimates used in the proof of
\cite[Lemma 4.3]{Breit2025}, together with
\eqref{eq:critical-local-maxreg}, give
\begin{equation}\label{eq:annular-data}
 \norm{K_m}{L_t^{a_0}L_x^{p_0}(\mathcal C)}\le C.
\end{equation}

Rewrite the localized equation with the flat Stokes operator:
\begin{equation}\label{eq:localized-flat}
\begin{aligned}
 \partial_tU_m-\Delta U_m+\nabla P_m
 &=\zeta J_mg_m-\lambda_m\zeta(B_m\nabla v_m)v_m
   +K_m+R_m^{\rm c},\\
 \diver U_m&=H_m^{\rm a}+H_m^{\rm c},
\end{aligned}
\end{equation}
where
\begin{equation}\label{eq:coefficient-errors}
\begin{aligned}
 R_m^{\rm c}
 &:=(1-J_m)\partial_tU_m
   +\diver((A_m-\Id)\nabla U_m)
   +\diver((\Id-B_m)P_m),\\
 H_m^{\rm c}&:=(\Id-B_m^\top):\nabla U_m.
\end{aligned}
\end{equation}
Let $\kappa_m$ be the sum of the coefficient norms in
\eqref{eq:flat-coefficients} at the critical level. Then
$\kappa_m\to0$. The multiplier estimates in the proof of
\cite[Lemma 4.2]{Breit2025} and \eqref{eq:critical-local-maxreg} imply
\begin{equation}\label{eq:coefficient-errors-small}
 \norm{R_m^{\rm c}}{L_t^{a_0}L_x^{p_0}}\le C\kappa_m.
\end{equation}
More explicitly, the product estimates give
\[
 \norm{R_m^{\rm c}}{L_t^{a_0}L_x^{p_0}}
 \le C\kappa_m\left(
 \norm{\partial_tU_m}{L_t^{a_0}L_x^{p_0}}
 +\norm{\nabla U_m}{L_t^{a_0}W_x^{1,p_0}}
 +\norm{P_m}{L_t^{a_0}W_x^{1,p_0}}
 \right).
\]
The last norm is controlled because the pressure was first normalized by its
mean on $B_{2/3}^+$ and then multiplied by a compactly supported cutoff.

The two divergence data in \eqref{eq:localized-flat} are separately
admissible for \Cref{lem:flat-nonzero-divergence}. Indeed,
\eqref{eq:piola} and the compact support of $U_m$ give the exact flux
identities
\begin{equation}\label{eq:divergence-fluxes}
 H_m^{\rm a}=\diver(B_mU_m),\qquad
 H_m^{\rm c}=\diver((\Id-B_m)U_m).
\end{equation}
Here the first identity uses
$B_m^\top:\nabla v_m=0$, while the second is just
$\diver U_m-B_m^\top:\nabla U_m$. Hence, for almost every $t$,
\begin{equation}\label{eq:separate-compatibility}
 \int_{\Hh}H_m^{\rm a}(t,x)\dd x=0,
 \qquad
 \int_{\Hh}H_m^{\rm c}(t,x)\dd x=0.
\end{equation}
These are exactly the spatial compatibility conditions in
\Cref{lem:flat-nonzero-divergence}. Since the coefficients are independent
of time,
\begin{equation}\label{eq:divergence-time-derivatives}
 \partial_tH_m^{\rm a}=\diver(B_m\partial_tU_m),\qquad
 \partial_tH_m^{\rm c}=\diver((\Id-B_m)\partial_tU_m)
 \quad\text{in }\mathcal D'(\Hh).
\end{equation}
The $L^p\to\dot W^{-1,p}$ boundedness of divergence, the critical local
estimate, and the multiplier product rule therefore yield directly
\begin{equation}\label{eq:divergence-data-audit}
\begin{aligned}
 &\norm{\nabla H_m^{\rm a}}{L_t^{a_0}L_x^{p_0}}
 +\norm{\partial_tH_m^{\rm a}}
 {L_t^{a_0}\dot W_x^{-1,p_0}}\le C,\\
 &\norm{\nabla H_m^{\rm c}}{L_t^{a_0}L_x^{p_0}}
 +\norm{\partial_tH_m^{\rm c}}
 {L_t^{a_0}\dot W_x^{-1,p_0}}\le C\kappa_m.
\end{aligned}
\end{equation}
Here the first spatial derivative estimate also uses that
$H_m^{\rm a}=B_m^\top:(v_m\otimes\nabla\zeta)$, whereas the second follows
by applying the $W^{1,p_0}$ multiplier norm of $\Id-B_m$ to $\nabla U_m$.
Because $\zeta$ vanishes near the lower endpoint of $\mathcal I$, both data
have zero initial trace. Thus every compatibility and regularity hypothesis
of the flat theorem is verified separately, without an annular mean
correction.

\smallskip
\noindent\emph{Step 3: the three flat Stokes problems.}
On $\mathcal C$, with zero initial value and zero trace on $\partial\Hh$,
solve all three flat Stokes problems first at the critical pair
$(a_0,p_0)$, corresponding to the following momentum and divergence data:
\begin{align*}
 (F_m^f,D_m^f)
   &=(\zeta J_mg_m,0),\\
 (F_m^{\rm c},D_m^{\rm c})
   &=\big(-\lambda_m\zeta(B_m\nabla v_m)v_m+R_m^{\rm c},
          H_m^{\rm c}\big),\\
 (F_m^h,D_m^h)&=(K_m,H_m^{\rm a}).
\end{align*}
Denote the solutions by $(U_m^i,P_m^i)$,
$i\in\{f,{\rm c},h\}$. By
\Cref{lem:flat-nonzero-divergence},
\eqref{eq:coefficient-errors-small}, and
\eqref{eq:divergence-data-audit}, all three problems belong to the same
local-$L^{p_0}$ homogeneous uniqueness class. The localized pair
$(U_m,P_m)$ belongs to that class as well. Linearity and uniqueness therefore
give
\begin{equation}\label{eq:flat-sum}
 U_m=U_m^{\rm c}+U_m^f+U_m^h,\qquad
 P_m=P_m^{\rm c}+P_m^f+P_m^h+c_m(t).
\end{equation}

The forced data belong simultaneously to the critical class and to
$L^{P_f}(\mathcal C)$ because they have fixed compact support. By the
cross-exponent consistency in \Cref{lem:flat-nonzero-divergence}, the critical
realization $(U_m^f,P_m^f)$ is the same realization as the one obtained with
$(a,p)=(P_f,P_f)$. Consequently maximal regularity at $P_f$ yields
\begin{equation}\label{eq:forced-flat-estimate}
 \norm{\nabla P_m^f}{L^{P_f}(\mathcal C)}
 \le C\norm{\zeta J_mg_m}{L^{P_f}(\mathcal C)}
 \le C.
\end{equation}
For the critical part, \eqref{eq:nonlinear-forcing} and
\eqref{eq:coefficient-errors-small} give
\begin{equation}\label{eq:critical-flat-estimate}
 \norm{\nabla P_m^{\rm c}}{L_t^{a_0}L_x^{p_0}(\mathcal C)}
 \le C(\lambda_m+\kappa_m)\longrightarrow0.
\end{equation}
Choosing at each time the homogeneous Sobolev representative of
$P_m^{\rm c}$, the critical embedding \eqref{eq:critical-embedding} implies
\begin{equation}\RtwoStart
\label{eq:critical-pressure-strong}
 \norm{P_m^{\rm c}}{L^{5/3}(\mathcal C)}
 \le C\norm{\nabla P_m^{\rm c}}
 {L_t^{5/3}L_x^{15/14}(\mathcal C)}
 \longrightarrow0.
\end{equation}

For a fixed contraction radius $\tau$, this also makes its pressure-excess
contribution tend to zero: the averaging factors depend on $\tau$ but not
on $m$. The limit $m\to\infty$ is taken with $\tau$ fixed. Neither this
argument nor \eqref{eq:forced-flat-estimate} asserts strong convergence of
the forced pressure when the forces converge only weakly.
\RevEnd

Finally, \eqref{eq:annular-data} and \eqref{eq:divergence-data-audit}
give a uniform critical maximal-regularity bound for $(U_m^h,P_m^h)$. The data
$(K_m,H_m^{\rm a})$ vanish in $Q_{5/8}^+$, so this pair solves the
homogeneous flat Stokes system there. Repeated application of the local flat
estimate \cite[Lemma 4.3 with $\varphi=0$]{Breit2025} on nested
half-cylinders, normalizing the pressure by its spatial mean at each stage,
yields, for every finite $Q$,
\begin{equation}\label{eq:homogeneous-bootstrap}
 \norm{\nabla P_m^h}{L_t^{5/3}L_x^Q(Q_{1/2}^+)}
 \le C_Q.
\end{equation}
\RtwoStart
The spatial exponent ladder can be made explicit. Starting from
$p_0=15/14$, spatial Sobolev embedding improves the lower-order norms
$\nabla U_m^h$ and the spatially normalized $P_m^h$ to exponent
$3p_0/(3-p_0)=5/3$. The local flat estimate, with zero forcing and
zero divergence on a smaller cylinder, then controls their full
maximal-regularity norms at exponent $5/3$. Repeating the embedding gives
$3(5/3)/(3-5/3)=15/4$. At this exponent, which exceeds three,
$W^{1,15/4}$ embeds into $L^\infty$ on the spatial half-ball, hence into
every finite $L^Q$. A final application of the same local flat estimate
on nested cylinders gives \eqref{eq:homogeneous-bootstrap}. All steps
retain the time exponent $5/3$. The finite sequence of radii lies strictly
between $1/2$ and $5/8$, and its constants may depend on $Q$ but not on $m$.
No high-exponent estimate for the curved coefficients enters this ladder.
\RevEnd
Since $\zeta=1$ on $Q_{1/2}^+$, the pressure identity in
\eqref{eq:flat-sum}, restricted to that cylinder, is precisely
\eqref{eq:pressure-split}. Estimates
\eqref{eq:critical-pressure-strong},
\eqref{eq:forced-flat-estimate}, and
\eqref{eq:homogeneous-bootstrap} prove
\eqref{eq:critical-small}--\eqref{eq:hom-pressure}.
\end{proof}

\begin{remark}[Why no annular Bogovskii correction is needed]
The localized flat Stokes theorem allows nonzero divergence data with
$\nabla D\in L^{a_0}L^{p_0}$ and
$\partial_tD\in L^{a_0}\dot W^{-1,p_0}$, zero spatial mean, and zero initial
matching. The Piola identity gives the exact flux representations
\eqref{eq:divergence-fluxes}; these imply both the time-derivative bounds in
\eqref{eq:divergence-data-audit} and the separate zero means in
\eqref{eq:separate-compatibility}. A Bogovskii velocity could be
introduced, but it would only repackage the same estimates and is not needed
for the decomposition.
\end{remark}

\section{Flat Stokes decay exponents}\label{sec:flat}

The two exponents in \Cref{lem:one-step} follow from elementary scaling once
the pressure separation is available.

\begin{lemma}[Flat velocity and pressure decay]\label{lem:flat-decay}
Let $(v,q)$ solve
\[
 \partial_t v-\Delta v+\nabla q=g,\qquad \diver v=0
 \quad\text{in }Q_{1/2}^+,\qquad v=0\quad\text{on }\{x_3=0\},
\]
with $g\in L^{P_f}$ and with the energy and $L^{5/3}$ pressure norms bounded
by $M$, and assume also $\norm{g}{L^{P_f}}\le M$. Then, for every
\begin{equation}\label{eq:velocity-decay-exponent}
 0<\nu<\min\left\{6-\frac{15}{P_f},\frac{12}{5}\right\},
\end{equation}
there is $C=C(M,P_f,\nu)$ such that
\begin{equation}\label{eq:flat-velocity-decay}
 {\fintx}_{Q_\tau^+}|v|^3\le C\tau^\nu,
 \qquad 0<\tau<\frac14.
\end{equation}

Independently, let $q^f,q^h$ satisfy
\[
 \nabla q^f\in L^{P_f}(Q_{1/2}^+),\qquad
 \nabla q^h\in L_t^{5/3}L_x^Q(Q_{1/2}^+),
\]
where $P_f>5/2$ and $Q>15/4$. Then for $0<\tau<1/4$,
\begin{align}
 &\tau^3\left(
 {\fintx}_{I_\tau}{\fintx}_{B_\tau^+}
 |q^f-(q^f)_{B_\tau^+}|^{5/3}
 \right)^{9/5}
 \le C\tau^{6-15/P_f}\norm{\nabla q^f}{L^{P_f}}^3,
 \label{eq:forced-decay}\\
 &\tau^3\left(
 {\fintx}_{I_\tau}{\fintx}_{B_\tau^+}
 |q^h-(q^h)_{B_\tau^+}|^{5/3}
 \right)^{9/5}
 \le C_Q\tau^{12/5-9/Q}
 \norm{\nabla q^h}{L_t^{5/3}L_x^Q}^3.
 \label{eq:hom-decay}
\end{align}
\end{lemma}

\begin{proof}
Choose a space--time cutoff $\chi$ which is one on $Q_{3/8}^+$ and is
supported in $Q_{1/2}^+$. Solve the flat half-space Stokes system with
forcing $\chi g$, zero divergence, zero boundary trace and zero initial
value, and denote its solution by $(v^f,q^f)$. By
\Cref{lem:flat-nonzero-divergence} with $D=0$ and $a=p=P_f$, flat maximal
regularity gives
\begin{equation}\label{eq:flat-velocity-forced}
 \norm{\partial_tv^f}{L^{P_f}}
 +\norm{\nabla^2v^f}{L^{P_f}}
 +\norm{\nabla q^f}{L^{P_f}}
 \le C\norm{g}{L^{P_f}}.
\end{equation}
Thus the parabolic Sobolev embedding yields
$v^f\in C_{\para}^{0,\alpha}$ for every
$\alpha<\min\{2-5/P_f,1\}$.

The remainder $v^h:=v-v^f$ solves a homogeneous flat Stokes system in
$Q_{3/8}^+$. Starting from the assumed energy and pressure bounds, repeated
local boundary estimates
\cite[Lemma 4.3 with $\varphi=0$]{Breit2025} on nested half-cylinders give
\[
 v^h\in W_t^{1,5/3}L_x^Q\cap L_t^{5/3}W_x^{2,Q}
 \qquad\text{for every finite }Q.
\]
The mixed parabolic embedding gives
$v^h\in C_{\para}^{0,\alpha}$ for
$\alpha<4/5-3/Q$. Taking $Q$ large yields every $\alpha<4/5$;
the time exponent $5/3$ is retained throughout this bootstrap. Both pieces
have zero trace on the flat boundary. Consequently, in a smaller half-cylinder,
\[
 |v(t,x)|\le Cx_3^\alpha
 \qquad\text{for every }
 \alpha<\min\left\{2-\frac5{P_f},\frac45\right\}.
\]
Averaging the cube and choosing $3\alpha>\nu$ proves
\eqref{eq:flat-velocity-decay}. Notice that the decomposition is essential:
the local estimate used for the homogeneous remainder raises spatial
regularity but, by itself, does not change the time exponent from $5/3$ to
$P_f$.

Spatial Poincar\'e followed by H\"older in the five-dimensional parabolic
cylinder gives
\[
 \norm{q^f-(q^f)_{B_\tau^+}}{L^{5/3}(Q_\tau^+)}
 \le C\tau^{4-5/P_f}\norm{\nabla q^f}{L^{P_f}(Q_\tau^+)}.
\]
Since the pressure part of the excess equals $\tau^{-6}$ times the cube of the
$L^{5/3}$ norm, this is \eqref{eq:forced-decay}. For the homogeneous part,
perform Poincar\'e and H\"older only in space:
\[
 \norm{q^h-(q^h)_{B_\tau^+}}{L_t^{5/3}L_x^{5/3}(Q_\tau^+)}
 \le C\tau^{14/5-3/Q}
 \norm{\nabla q^h}{L_t^{5/3}L_x^Q(Q_\tau^+)}.
\]
Cubing and multiplying by $\tau^{-6}$ gives \eqref{eq:hom-decay}.
\end{proof}

\begin{remark}
As $Q\to\infty$, the homogeneous pressure exponent approaches $12/5$. Thus
one may choose any final excess exponent
\begin{equation}\label{eq:beta-max}
 0<\beta<\min\left\{6-\frac{15}{P_f},\frac{12}{5}\right\}.
\end{equation}
The boundary exponent $P_b$ does not occur in \eqref{eq:beta-max}.
More explicitly, given $\alpha$ in \eqref{eq:holder-range}, choose numbers
and a finite $Q$ so that
\[
 3\alpha<\beta<\theta<
 \min\left\{6-\frac{15}{P_f},\frac{12}{5}-\frac9Q\right\}.
\]
This order of choices is used in the iteration below.
\end{remark}

\section{Physical-scale iteration}\label{sec:iteration}

Scaling \eqref{eq:one-step} back to a cylinder of radius $r$, and absorbing
the fixed factor $|Q_1^+|^{3/P_f}$ from \eqref{eq:excess-scaling} into the
constants and smallness thresholds, gives
\begin{equation}\label{eq:physical-one-step}
 r^3E_{\tau r}(u,\pi)
 \le C_0\tau^\theta\left[
 r^3E_r(u,\pi)
 +r^9\left({\fintx}_{Q_r}|f|^{P_f}\right)^{3/P_f}
 \right].
\end{equation}
After division by $r^3$,
\begin{equation}\label{eq:physical-one-step-2}
 E_{\tau r}\le C_0\tau^\theta\big(E_r+\mathfrak F(r)\big),
 \qquad
 \mathfrak F(r):=r^6\left({\fintx}_{Q_r}|f|^{P_f}\right)^{3/P_f}.
\end{equation}
Since $|Q_r|\simeq r^5$,
\begin{equation}\label{eq:force-modulus}
 \mathfrak F(r)\le Cr^{\mu_f}\norm{f}{L^{P_f}(I\times\Omega)}^3,
 \qquad \mu_f=6-\frac{15}{P_f}.
\end{equation}

Choose $0<\beta<\min\{\theta,\mu_f\}$ and then fix $\tau$ so small that
\begin{equation}\label{eq:contraction}
 C_0\tau^\theta\le\frac12\tau^\beta.
\end{equation}
By \Cref{cor:uniform-threshold}, all boundary charts satisfy the geometric
smallness condition below a common radius $R_0$. The unit-scale smallness
hypothesis at a physical radius $r$ is, by
\eqref{eq:excess-scaling},
\begin{equation}\label{eq:physical-smallness}
 r^3\bigl(E_r+\mathfrak F(r)\bigr)\le\eps_*.
\end{equation}
Choose $R_f>0$ so that
$r^{3+\mu_f}\norm{f}{L^{P_f}}^3$ is below the force threshold whenever
$r<R_f$. Suppose that, at a fixed boundary point, there is an
$R<\min\{R_0,R_f\}$ for which $R^3E_R$ is below the fluid threshold. Then
\eqref{eq:physical-one-step-2} and \eqref{eq:force-modulus} propagate
\eqref{eq:physical-smallness} to every radius $\tau^kR$. Writing
$D_k=E_{\tau^kR}$ gives
\begin{equation}\label{eq:recurrence}
 D_{k+1}\le\frac12\tau^\beta D_k
 +C\tau^\theta R^{\mu_f}\tau^{k\mu_f}\norm{f}{L^{P_f}}^3.
\end{equation}
A discrete convolution yields
\begin{equation}\label{eq:iterated-decay}
 E_{\tau^kR}(u,\pi)
 \le C\tau^{k\beta}
 \left(E_R(u,\pi)+R^{\mu_f}\norm{f}{L^{P_f}}^3\right).
\end{equation}
Interpolation between adjacent discrete radii gives the same estimate for all
$0<r<R$. We next make explicit why a boundary-centred estimate yields a
Campanato estimate at every nearby interior point.

For a boundary point $\widehat z=(\widehat t,\widehat x)$, let
$E_r^{\rm b}(\widehat z)$ denote \eqref{eq:excess} in the tangent-plane chart
at $\widehat x$ and let
\[
 \mathcal A_r^{\rm b}(\widehat z)
 :=r^3\bigl(E_r^{\rm b}(\widehat z)+\mathfrak F_{\widehat z}(r)\bigr).
\]
The finite atlas, the tangent-plane regraphing lemma and the uniform
bi-Lipschitz bounds imply that physical cylinders and their flattened images
are comparable, with constants independent of $\widehat z$ and
$0<r<R_0$. We shall use this comparability without further comment.

\begin{lemma}[Stability of the boundary starting scale]
\label{lem:nearby-start}
There are $c_0\in(0,1/16)$ and $\eps_{\rm n}>0$, depending only on the
structural data and the thresholds in \Cref{lem:one-step}, with the following
property. If $z_0\in I\times\partial\Omega$, $0<R<R_0$ and
\begin{equation}\label{eq:start-at-z0}
 \mathcal A_R^{\rm b}(z_0)
 +R^{3+\mu_f}\norm{f}{L^{P_f}}^3\le\eps_{\rm n},
\end{equation}
then there exists $h>0$, with $h<c_0^2R^2$, such that every boundary point
\[
 \widehat z=(\widehat t,\widehat x),\qquad
 |\widehat x-x_0|<c_0R,\quad |\widehat t-t_0|<h,
 \quad\widehat t\in I,
\]
has an admissible starting scale $R/4$. The spatial constant $c_0$ is uniform;
the positive time width $h$ may also depend on the solution at the fixed
starting scale. Moreover, after decreasing
$\eps_{\rm n}$ once, the iteration gives
\begin{equation}\label{eq:uniform-boundary-decay}
 E_s^{\rm b}(\widehat z)
 \le C_Rs^\beta,\qquad 0<s<R/16,
\end{equation}
where $C_R$ is independent of $\widehat z$.
\end{lemma}

\begin{proof}
Choose $c_0$ according to the atlas constants. For
$\widehat t\le t_0$ and $|\widehat t-t_0|<c_0^2R^2$, the physical cylinder
of radius $R/4$ at $\widehat z$ is contained in the physical cylinder of
radius $R$ at $z_0$. For measurable sets $A\subset B$ of comparable measure
and $1<p<\infty$,
\begin{equation}\label{eq:mean-comparison}
 \int_A|h-(h)_A|^p
 \le 2^p\int_B|h-(h)_B|^p.
\end{equation}
Applying this inequality at almost every time to the pressure, and using the
corresponding elementary inclusion estimate for the velocity, gives
\[
 \mathcal A_{R/4}^{\rm b}(\widehat z)
 \le C\mathcal A_R^{\rm b}(z_0)
 +C R^{3+\mu_f}\norm{f}{L^{P_f}}^3.
\]
Choose $\eps_{\rm n}$ so that these terms use at most half of the
admissible threshold. If $\widehat t>t_0$, the only additional portion lies
in a future strip $(t_0,t_0+h)\times B_{CR}(x_0)$. At the fixed radius $R$,
its contribution is bounded by
\[
 C R^{-2}\int_{(t_0,t_0+h)\times(\Omega\cap B_{CR}(x_0))}|u|^3
 +C R^{-3}\left(
 \int_{(t_0,t_0+h)\times(\Omega\cap B_{CR}(x_0))}|\pi|^{5/3}
 \right)^{9/5}.
\]
Both terms tend to zero as $h\downarrow0$, by absolute continuity of the
integral. The use of unnormalized pressure on this strip is legitimate by
\eqref{eq:mean-comparison}, applied separately at each time. Reduce $h$ so
that the strip contribution uses the remaining margin. Geometry is already
below its independent threshold because $R<R_0$. Thus
\eqref{eq:iterated-decay}, applied at every
$\widehat z$, and interpolation between consecutive discrete radii yield
\eqref{eq:uniform-boundary-decay}. All constants are uniform over the finite
atlas.
\end{proof}

For an interior cylinder $Q_r^-(z)\Subset I\times\Omega$, define
\begin{equation}\label{eq:interior-excess}
\begin{aligned}
 \widetilde E_r(z)
 &:={\fintx}_{Q_r^-(z)}
 |u-(u)_{Q_r^-(z)}|^3\\
 &\quad+r^3\left(
 {\fintx}_{I_r^-(t)}{\fintx}_{B_r(x)}
 |\pi-(\pi)_{B_r(x)}(\sigma)|^{5/3}\dd y\dd\sigma
 \right)^{9/5}.
\end{aligned}
\end{equation}

\begin{lemma}[Interior excess decay]\label{lem:interior-decay}
Fix the parameters $\tau,\theta,\beta$ used above. There is
$\eps_{\rm int}>0$ such that, whenever
$Q_\rho^-(z)\Subset I\times\Omega$ and
\[
 \rho^3\left(\widetilde E_\rho(z)
 +|(u)_{Q_\rho^-(z)}|^3+\mathfrak F_z(\rho)\right)
 \le\eps_{\rm int},
\]
one has
\begin{equation}\label{eq:interior-iteration}
 \widetilde E_r(z)
 \le C\left(\frac r\rho\right)^\beta
 \left(\widetilde E_\rho(z)
 +\rho^{\mu_f}\norm{f}{L^{P_f}}^3\right),
 \qquad 0<r<\frac{\rho}{2}.
\end{equation}
\end{lemma}

\begin{proof}
Rescale first to $\rho=1$. Subtracting the space--time velocity average
introduces a constant drift in the normalized momentum equation. The added
mean bound controls this drift. The interior versions of the critical
pressure decomposition and flat estimates
\cite[Lemmas 4.4 and 4.5]{Breit2025} give the one-step oscillation estimate,
with constants uniform for the bounded drifts in question. A constant drift
can also be removed by a Galilean change of variables; the nested cylinders
remain uniformly comparable when the unit-scale mean is sufficiently small.
The usual interior suitable-solution energy estimate applies to the
mean-subtracted velocity. No boundary coefficient occurs.

We record why the mean condition persists rather than dropping it in the
iteration. Put $r_k=\tau^k\rho$, $a_k=(u)_{Q_{r_k}^-(z)}$ and
$b_k=r_k|a_k|$. Inclusion of the cylinders and H\"older give
\[
 b_{k+1}\le\tau b_k+
 C\tau^{-2/3}\bigl(r_k^3\widetilde E_{r_k}(z)\bigr)^{1/3}.
\]
Once the oscillation recurrence is applied, its rightmost term decays at
least as $\tau^{k(1+\beta/3)}$ times the initial amplitude and the force
amplitude. The discrete convolution therefore bounds $b_k$ uniformly by a
constant times their sum and makes $b_k\to0$. Choosing the initial
threshold sufficiently small closes the mean and oscillation inductions
together. For this induction use the force norm on the initial cylinder:
$\rho^{3+\mu_f}\|f\|_{L^{P_f}(Q_\rho^-(z))}^3$ is controlled, up to a
fixed volume factor, by the assumed $\rho^3\mathfrak F_z(\rho)$.
The force term decays with exponent $\mu_f>\beta$; only in the final bound
is its local norm enlarged to $\|f\|_{L^{P_f}(I\times\Omega)}$.
The same convolution as in
\eqref{eq:recurrence}--\eqref{eq:iterated-decay}, followed by comparison of
adjacent radii, proves \eqref{eq:interior-iteration}.
\end{proof}

\begin{proposition}[Interior--boundary Campanato patching]
\label{prop:patching}
Assume \eqref{eq:start-at-z0}. After decreasing the relative neighbourhood of
$z_0$, there is a constant $C_R$ such that
\begin{equation}\label{eq:patched-campanato}
 \widetilde E_r(z)\le C_Rr^\beta
\end{equation}
for every interior point $z$ in that neighbourhood and every sufficiently
small $r$. The same velocity-oscillation bound holds when the cylinder meets
the boundary.
\end{proposition}

\begin{proof}
Let $z=(t,x)$ be an interior point in a smaller cylinder supplied by
\Cref{lem:nearby-start}, put $d=\dist(x,\partial\Omega)$, and choose a nearest
boundary point $\widehat x$. If $r\ge d/8$, the cylinder at $z$ is contained
in a boundary cylinder of radius $c_1r$ centred at
$\widehat z=(t,\widehat x)$. By choosing the neighbourhood small enough,
$c_1r<R/16$, and \eqref{eq:mean-comparison} together with
\eqref{eq:uniform-boundary-decay} gives
\[
 \widetilde E_r(z)\le C E_{c_1r}^{\rm b}(\widehat z)
 \le C_Rr^\beta.
\]
The identical comparison, without subtracting the velocity average on the
right, covers cylinders meeting the boundary.

Suppose now that $r<d/8$ and set $\rho=d/4$. Then
$B_{2\rho}(x)\subset\Omega$. Comparison with the boundary cylinder of radius
$c_2d$ at $\widehat z$ yields
\begin{equation}\label{eq:interior-start-from-boundary}
 \widetilde E_\rho(z)\le C_Rd^\beta.
\end{equation}
Consequently,
\[
 \rho^3\left(\widetilde E_\rho(z)
 +|(u)_{Q_\rho^-(z)}|^3+\mathfrak F_z(\rho)\right)
 \le C_R\bigl(d^{3+\beta}+d^{3+\mu_f}\bigr).
\]
Here Jensen's inequality and the boundary bound for the unsubtracted
velocity give $|(u)_{Q_\rho^-(z)}|^3\le C_Rd^\beta$.
The right-hand side is below $\eps_{\rm int}$ throughout a sufficiently small
fixed neighbourhood of $z_0$. Applying \Cref{lem:interior-decay} and using
$\mu_f>\beta$ in \eqref{eq:interior-start-from-boundary} gives
\[
 \widetilde E_r(z)
 \le C_R(r/\rho)^\beta(\rho^\beta+\rho^{\mu_f})
 \le C_Rr^\beta.
\]
This proves \eqref{eq:patched-campanato} in both distance regimes.
\end{proof}

The velocity part of \eqref{eq:patched-campanato} is precisely
\[
 {\fintx}_{Q_r^-(z)\cap(I\times\Omega)}
 |u-(u)_{Q_r^-(z)\cap(I\times\Omega)}|^3\le C_Rr^\beta.
\]
Campanato's characterization on uniformly Lipschitz graph cylinders (see, for
example, \cite{GiaquintaMartinazzi}), followed by the bi-Lipschitz change of
coordinates, therefore gives
\begin{equation}\label{eq:holder-from-beta}
 u\in C^{0,\alpha}_{\para}\quad\text{for every }0<\alpha<\frac\beta3.
\end{equation}
Combining \eqref{eq:beta-max} and \eqref{eq:holder-from-beta} proves
\eqref{eq:holder-range}.

\section{The boundary singular set}\label{sec:singular}

Let $\mathcal R_\partial$ be the set of boundary points at which
\eqref{eq:start-at-z0} holds at some sufficiently small starting radius,
and define
\begin{equation}\label{eq:singular-set-definition}
 \Sigma:=(I\times\partial\Omega)\setminus\mathcal R_\partial.
\end{equation}
Lemma~\ref{lem:nearby-start} and its uniform decay show that nearby boundary
points also satisfy \eqref{eq:start-at-z0} after reducing their starting
radius. Thus $\mathcal R_\partial$ is relatively open and $\Sigma$ is
relatively closed. Proposition~\ref{prop:patching} gives local regularity
there. At each such point, \eqref{eq:iterated-decay} makes $r^3E_r\to0$.
The iteration can consequently be restarted with the thresholds for any
other admissible excess exponent $\beta$. The same regular set therefore
supports every H\"older exponent in \eqref{eq:holder-range}.

Once $r<R_0$, geometric smallness is automatic. The rescaled force
$r^3\mathfrak F_{z_0}(r)$ tends to zero uniformly by
\eqref{eq:force-modulus}. If \eqref{eq:physical-smallness} held at some
sufficiently small radius about $z_0$, then \eqref{eq:iterated-decay} would
give \eqref{eq:start-at-z0} at a smaller radius, and hence
$z_0\in\mathcal R_\partial$. Consequently, for every $z_0\in\Sigma$ and all
sufficiently small $r$,
\begin{equation}\label{eq:failed-smallness}
 r^3E_r^{\rm b}(z_0)\ge c\eps_*.
\end{equation}
In physical coordinates define
\begin{align}
 \mathcal C_u(z_0,r)&:=r^{-2}\int_{Q_r(z_0)\cap(I\times\Omega)}|u|^3,
 \label{eq:Cu}\\
 \mathcal C_\pi(z_0,r)&:=
 \left(r^{-5/3}\int_{Q_r(z_0)\cap(I\times\Omega)}|\pi|^{5/3}\right)^{9/5}.
 \label{eq:Cp}
\end{align}
Uniform chart comparability and \eqref{eq:failed-smallness} imply that there
is $c\eps_*>0$ such that, at a singular point and every sufficiently small
scale,
\begin{equation}\label{eq:bad-density}
 \mathcal C_u(z_0,r)+\mathcal C_\pi(z_0,r)\ge c\eps_*.
\end{equation}
Here the pressure in $E_r$ is normalized by its spatial mean, while
\eqref{eq:Cp} uses the unnormalized pressure; the elementary
mean-subtraction inequality bounds the former by a constant times the latter.
If the velocity alternative in \eqref{eq:bad-density} occurs, then
\[
 \int_{Q_r}|u|^3\ge c\eps_*r^2.
\]
Since $|Q_r|\simeq r^5$, H\"older's inequality gives
\begin{equation}\label{eq:u-density}
 \int_{Q_r(z_0)}|u|^{10/3}
 \ge c\eps_*^{10/9}r^{5/3}.
\end{equation}
The pressure alternative directly yields
\begin{equation}\label{eq:p-density}
 \int_{Q_r(z_0)}|\pi|^{5/3}
 \ge c\eps_*^{5/9}r^{5/3}.
\end{equation}
Thus every singular point and every sufficiently small scale satisfy
\begin{equation}\label{eq:joint-density}
 \int_{Q_r(z_0)}\left(|u|^{10/3}+|\pi|^{5/3}\right)
 \ge c_{\eps_*}r^{5/3}.
\end{equation}

We now spell out the covering step. On space--time set
\begin{equation}\label{eq:parabolic-metric}
 d_{\para}\bigl((t,x),(s,y)\bigr)
 :=\max\{|x-y|,|t-s|^{1/2}\},
 \qquad
 \mathcal P_r(t,x):=(t-r^2,t+r^2)\times B_r(x).
\end{equation}
The sets $\mathcal P_r$ are balls up to an immaterial constant for
$d_{\para}$. We use them, rather than backward cylinders, in the covering
selection. This point is important: a fixed dilation of a selected backward
cylinder need not cover a later backward cylinder that intersects it.

\begin{proposition}[Parabolic fine covering]\label{prop:vitali-covering}
Let $K\subset\Sigma$ be compact. Then
\[
 \mathcal H_{\para}^{5/3}(K)=0.
\]
\end{proposition}

\begin{proof}
Extend
\[
 g:=|u|^{10/3}+|\pi|^{5/3}
\]
by zero outside $I\times\Omega$. The energy interpolation inequality gives
$u\in L^{10/3}(I\times\Omega)$. Moreover,
\[
 W^{1,15/14}(\Omega)\hookrightarrow L^{5/3}(\Omega),
\]
so \Cref{def:suitable} gives $\pi\in L^{5/3}(I\times\Omega)$. Hence
$g\in L^1(\R\times\R^3)$.

Fix $\delta>0$. For every $z\in K$, choose a radius
$r_z<\delta$ small enough that \eqref{eq:joint-density} holds. In fact the
same is true for every still smaller radius, so these balls form a fine cover
of $K$. The metric $5r$-covering lemma supplies a countable pairwise disjoint
subfamily
\[
 \{\mathcal P_{r_i}(z_i)\}_{i\in\N},
 \qquad
 K\subset\bigcup_i\mathcal P_{5r_i}(z_i).
\]
Since the backward cylinder $Q_{r_i}^-(z_i)$ is contained in
$\mathcal P_{r_i}(z_i)$, \eqref{eq:joint-density} and disjointness give
\begin{align}
 \mathcal H_{\para,C\delta}^{5/3}(K)
 &\le C\sum_i r_i^{5/3}\nonumber\\
 &\le C_{\eps_*}\sum_i\int_{Q_{r_i}^-(z_i)}g
 \le C_{\eps_*}\int_{N_\delta^{\para}(K)}g.
 \label{eq:vitali-estimate}
\end{align}
Here $N_\delta^{\para}(K)$ is the $d_{\para}$-neighbourhood of $K$; changing
$\delta$ by an absolute constant does not affect the conclusion.

The Lipschitz boundary has three-dimensional Lebesgue measure zero, hence
$K\subset I\times\partial\Omega$ has four-dimensional space--time Lebesgue
measure zero. Therefore
$\lvert N_\delta^{\para}(K)\rvert\to0$ as $\delta\downarrow0$. Absolute
continuity of the integral of $g$ makes the right-hand side of
\eqref{eq:vitali-estimate} tend to zero. This proves the proposition.
\end{proof}

Finally, exhaust the relatively closed set $\Sigma$ by the compact sets
\[
 K_j:=\Sigma\cap\bigl([1/j,T-1/j]\times\partial\Omega\bigr)
\]
(discarding empty members). Countable subadditivity and
\Cref{prop:vitali-covering} yield
\begin{equation}\label{eq:H-zero}
 \mathcal H^{5/3}_{\para}(\Sigma)=0.
\end{equation}
This is the scale-counting mechanism used in
\cite[Theorem 2.7]{Breit2025}, with the covering geometry and absolute
continuity step made explicit. No geometric term enters
\eqref{eq:vitali-estimate}, because the multiplier threshold is already
uniformly satisfied below $R_0$.

\RtwoStart
\section{Limitations of an additive geometric normalization}\label{sec:dimension}
Consider the particular normalization
\begin{equation}\label{eq:wrong-normalization}
 \lambda_r^3=r^3E_r
 +r^9\left({\fintx}_{Q_r}|f|^{P_f}\right)^{3/P_f}+\eta_r^3.
\end{equation}
Scaling back a unit-scale estimate with this right-hand side introduces
$r^{-3}\eta_r^3$ into the physical-scale recurrence. The available estimate
$\eta_r\lesssim r^{\gamma_b}$ gives only
\begin{equation}\label{eq:wrong-geometry-scale}
 r^{-3}\eta_r^3\lesssim r^{3\gamma_b-3}.
\end{equation}
For finite $P_b>3$ the exponent on the right is negative. This upper bound
therefore supplies no decay with which to close our Campanato recurrence.
It does not prove divergence of the geometric term itself, a dimensional
contradiction, or impossibility of every additive geometric excess.
Our alternative is the separate admissibility condition
\begin{equation}\label{eq:geometry-role}
 \eta_r\le\delta_*(\tau),
\end{equation}
which holds below a uniform starting radius. The proof of
\Cref{lem:one-step}, rather than the failure of the bound in
\eqref{eq:wrong-geometry-scale}, establishes the validity of this choice.

\RevEnd\section{Discussion}\label{sec:discussion}

The analysis separates the boundary exponent from the final regularity
exponent. The condition $P_b>3$ supplies a positive geometric decay
$\gamma_b$ and hence a uniform starting radius. The force exponent and the
flat homogeneous pressure determine the Campanato rate. This explains why a
term such as $3(1-3/P_b)$ should not appear in the final minimum of decay
exponents.

The pressure separation is local and does not require a domain-uniform
nonlocal boundary Caccioppoli estimate. It first obtains only the critical rough-coefficient estimate,
then rewrites the localized equation as a flat Stokes system. The rough
coefficient errors disappear in
$L_t^{5/3}W_x^{1,15/14}\hookrightarrow L^{5/3}$, whereas the cutoff
commutator is moved into a flat homogeneous problem before any exponent is
raised. Thus the argument never asks the multiplier boundary to support
$W_x^{2,Q}$ regularity for $Q>15/4$.

The Piola identity is also more than a notational convenience. It writes the
coefficient-divergence error and the annular cutoff divergence as two exact
compactly supported fluxes. Consequently their time derivatives lie in the
required homogeneous negative Sobolev space, their initial traces vanish, and
they are separately mean-free. This permits the three flat Stokes problems in
\Cref{prop:pressure-separation} to be solved in one uniqueness class without
transferring a time-dependent mean between them. It preserves both the support
property of the homogeneous data and the smallness of the critical data.

The final local-to-global step uses two different geometric decompositions.
For regularity, the relevant dichotomy is whether a cylinder radius is larger
or smaller than the distance of its centre to the boundary; this is what
allows boundary-centred and interior excess estimates to meet without losing
the exponent. For the singular set, the relevant objects are symmetric
parabolic metric balls, even though the PDE density estimate is obtained on
backward cylinders. Separating those roles removes a minor but genuine
covering ambiguity.

The proof covers the range $P_b>3$ associated with the positive H\"older differentiability exponent:
geometry supplies the uniform starting scale, critical flat localization
removes the rough coefficients before any exponent is raised, and the
interior--boundary patching converts the boundary decay into a full relative
neighbourhood regularity statement. The pressure separation and the
parabolic covering are formulated as independent propositions so that their
hypotheses can be checked without referring to the contradiction sequence.
\RtwoStart
At the endpoint $P_b=3$, the positive geometric decay used to reach the
uniform smallness threshold is lost; that case requires a separate argument.

More precisely, the present argument uses both $\gamma_b>0$ and a positive
Sobolev--Besov index gap. At $P_b=3$ the embedding into
$C^{1,\gamma_b}$ used for tangent-plane normalization no longer supplies
continuous tangent planes, and $\varepsilon_b=\gamma_b/2$ is unavailable.
We consequently claim neither the endpoint nor uniform constants as
$P_b\downarrow3$. An endpoint result would need new control of tangent
planes and of the required multiplier norms, for example through suitable
additional vanishing or summability assumptions, which are not proved here.

Treating the limiting geometry separately from the normalized solution is
consistent with the broader compactness and blow-up methodology; compare
the fluid compactness argument in \cite{Lin1998} and the regularity methods
surveyed in \cite{GiaquintaMartinazzi}. These references provide
methodological context, not the specific curved--flat pressure decomposition
proved here. No priority claim is made for the general idea of geometric
flattening or independent coefficient smallness, including in free-boundary
problems. The specific contribution is the two-level vanishing estimate and
its use to avoid raising the rough-coefficient Stokes exponent.

The finite atlas provides uniform radii and constants for regraphing,
flattening, and the starting scale. The local decay argument uses a single
chart, so its mechanism is local. On an unbounded domain with a uniformly
controlled atlas one may expect analogous local estimates, but a global
existence theorem, behavior at infinity, and uniform patching require
separate verification. We do not claim an extension to arbitrary unbounded
domains or to boundaries lacking the tangent-plane and multiplier controls.

For other boundary scales, the proof identifies concrete requirements:
small rescaled Lipschitz norms, and small multiplier norms at both levels
in \eqref{eq:two-levels}, with compatible regraphing and extension estimates.
The Besov fine index affects summability when the differential-index gap
closes. Lorentz refinements would require the corresponding product,
embedding and Stokes estimates, not merely replacing an $L^p$ symbol.
Thus the present product lemma suggests what must be checked for such
extensions but does not establish a Besov or Lorentz endpoint theorem.
\RevEnd

\section*{Acknowledgments}

During preparation of this submission version, ChatGPT (OpenAI) was used to assist the review and revision of proof arguments and mathematical estimates, reference checks, manuscript revision and language editing, and LaTeX preparation. The author takes full responsibility for independently verifying all mathematical arguments, references, and the final text.

\section*{Ethical and legal declarations}

\subsection*{Research ethics}
Not applicable. This theoretical mathematical study involves no human
participants or animal experiments.

\subsection*{Informed consent}
Not applicable.

\subsection*{Author contributions}
The author has accepted responsibility for the entire content of this
manuscript and approved its submission.

\subsection*{Use of large language models, AI and machine learning tools}
During preparation of this submission version, ChatGPT (OpenAI) was used to assist the review and revision of proof arguments and mathematical estimates, reference checks, manuscript revision and language editing, and LaTeX preparation. The author takes full responsibility for independently verifying all mathematical arguments, references, and the final text.

\subsection*{Conflict of interest}
The author declares no conflict of interest.

\subsection*{Research funding}
None declared.

\subsection*{Data availability}
Not applicable. No datasets were generated or analyzed in this theoretical
mathematical study.

\end{document}